\documentclass{amsart}

\usepackage{pstricks,geometry,amsmath,amssymb,epsfig,graphicx,graphics}

\newtheorem{Th}{Theorem}
\newtheorem{Lem}{Lemma}
\newtheorem{Prop}{Proposition}

\begin{document}

\title{Analysis of a mathematical model of syntrophic
bacteria in a chemostat} 

\author{Tewfik Sari, Miled EL-Hajji, J\'er\^ome Harmand} 

\address{Tewfik Sari, 
Universite de Haute Alsace, LMIA, 4 rue des Fr\`eres Lumi\`ere, 68093 Mulhouse, 
\& 
Inra-Inria Modemic research team, UMR
Mistea, SupaAgro, 2 place Viala, 34060 Montpellier, France.}
\email{Tewfik.Sari@uha.fr}
\address{
Miled EL Hajji,
ISSAT, Université de Kairouan
3000 Avenue Beit el hikma, Kairouan Tunisie}
\email{miled.elhajji@enit.rnu.tn}
\address{
J\'er\^ome Harmand,
LBE-INRA, UR050, Avenue des \'Etangs,  11100 Narbonne, France 
\& Inra-Inria Modemic research team, UMR
Mistea, SupaAgro, 2 place Viala, 34060 Montpellier, France.}
\email{harmand@supagro.inra.fr}

\date{\today}

\begin{abstract}
A mathematical model involving a syntrophic relationship between two populations 
of bacteria in a continuous culture is proposed. A detailed qualitative analysis is 
carried out. The local and global stability analysis of the equilibria are performed.
We demonstrate, under general assumptions of monotonicity, relevant from
an applied point of view, the asymptotic stability of the positive 
equilibrium point which corresponds to the coexistence of the two bacteria. 
A syntrophic relationship in the anaerobic digestion process is proposed as a real candidate for this model.
\end{abstract} 
\maketitle

\section{Introduction}
A synthrophic relationship between two organisms refers to a situation where
the species exhibit mutualistism but where, at the opposite of
what happens in a purely symbiotic relationship, one of the species can
grow without the
other. Such a situation can be mathematically formalized as follows.
Assume that a first
species denoted $X_1$ grows on a substrate $S_1$ forming an intermediate
product $S_2$.
This intermediate product is required by a second species $X_2$ to grow.
The limiting substrate of the second bacteria being the product of the
first bioreaction,
the second bacteria cannot grow if the first one is not present.

Such interactions are quite common in nature: it is why a number of
models have already been proposed in the literature.
Katsuyama {\em et al.} \cite{katsuyama}, proposed a model involving two mutualistic species
for describing pesticide degradation, 
while a more general case is considered by Kreikenbohm and Bohl \cite{bohl}.
Since mutualism involves generally species interacting through
intermediate products, other studies consider mutualistic relationships
in food webs. For instance, Bratbak and Thingstad \cite{bratbak}, or more recently,
Aota and Nakajima \cite{aota} considered the mutualism between phytoplankton and
bacteria through the carbon excretion by the phytoplankton. 
A model studied by Freedman {\em et al.} \cite{freedman}
was proposed to explain the observed coexistence of such species.
However, in the previous studies the models are very specific. In
particular, the mathematical analyses of the models are realized for
specific growth rates that are explicitely given (in most cases as Monod
functions). 

To extend the study of mutualism to more general systems, we have
recently considered more general assumptions notably with respect to the
growth rate functions considered in the models in using qualitative
hypotheses, cf. \cite{elhajji3} .
Furthermore, it was assumed that the species $X_1$ may be inhibited by
the product $S_2$
that it produces itself
while the species $X_2$ was simply limited by $S_2$. An example of such
interactions was given by the anaerobic digestion in which mutualistic
relationships allow certain classes of bacteria to coexist.
A mutualistic relation has been also considered in \cite{elhajji2}.
See \cite{elhajji1} for another model of coexistence in the chemostat. 

In the actual paper, following \cite{elhajji4}, we revisit the model 
proposed in \cite{elhajji3} in considering two main changes which 
significantly further extend the range of practical situations
covered by the model. 
First, we assume that there is some $S_2$ in the influent. In other
terms, the limiting substrate $S_2$ on which the species $X_2$ grows is
not only
produced by the species $X_1$ but is also available even if the species
$X_1$ is not present.
The second modification of the model is that the second species is
supposed to be
inhibited by an excess of $S_1$, the limiting substrate on which the
first species grows.
To illustrate the usefulness of such extensions of the original model by
El Hajji et al. \cite{elhajji3}, the biological interpretation of these
hypotheses within the context of the anaerobic
process is given in the appendix. 

The paper is organized as follows. In Section \ref{sec2},
we propose a modified system of four differential
equations from the original model in \cite{elhajji3}.
The positive equilibria are determined and their local and global
stability properties are established.
In the case when the system has a unique positive equilibrium,
the global asymptotic stability results are demonstrated through
the Dulac's criterion that rules out the possibility
of the existence of periodic solutions for the reduced
planar system, the Poincar\'{e}-Bendixon Theorem and the Butler-McGehee
Lemma. Hence, in this case, for every positive initial
conditions, the solutions converge to the positive equilibrium point
which corresponds to the coexistence of the two bacterial species as
observed in real
processes. Simulations are presented in Section \ref{sec3},
an example of a syntrophic relationship is given in Section \ref{sec4} as a
candidate for this model.
\section{Mathematical model}\label{sec2}
Let $S_1$, $X_1$, $S_2$ and $X_2$ denote, respectively, the concentrations
of the substrate, the first bacteria, the intermediate product, and the second bacteria present in the reactor at time $t$.
We neglect all species-specific death rates and take into account the dilution
rate only. Hence our model is described by the following system of ordinary
differential equations :
\begin{equation}
\label{model0}
\left\{
\begin{array}{rcl}
\dot{S}_1 & = & D (S^{in}_1 - S_1) - k_3 \mu_1(S_1,S_2) X_1\ ,\\
&&\\
\dot{X}_1 & = & \mu_1(S_1,S_2) X_1 - D X_1\  ,\\
&&\\
\dot{S_2} & = & D (S^{in}_2 -S_2)- k_2 \mu_2(S_1,S_2) X_2+k_1 \mu_1(S_1,S_2) X_1\  , \\
&&\\
\dot{X}_2 & = & \mu_2(S_1,S_2) X_2 - D X_2\ .\\
\end{array}
\right.
\end{equation}
Where
$S^{in}_1>0$ denotes the input concentration of substrate,
$S^{in}_2>0$ denotes the input concentration of the intermediate product and
$D>0$ is  the dilution rate.

Assume that the functional response of each species
$\displaystyle \mu_1,\mu_2: \mathbb{R}_+^2 \rightarrow \mathbb{R}_+ $
satisfies :
\begin{description}

\item[A1] $\mu_1,\mu_2: \mathbb{R}_+^2 \rightarrow \mathbb{R}_+$, of class $\mathcal{C}^1\ ,$

\item[A2] $\mu_1(0,S_2) = 0,\qquad\mu_2(S_1,0) = 0,\quad\forall \; (S_1,S_2) \in \mathbb{R}_+^2\ ,$

\item[A3] $\displaystyle \frac{\partial \mu_1}{\partial S_1}(S_1,S_2) > 0,
\quad
\displaystyle \frac{\partial \mu_1}{\partial S_2}(S_1,S_2) < 0,
\quad
\forall \; (S_1,S_2) \in \mathbb{R}_+^2\ ,$

\item[A4] $\displaystyle \frac{\partial \mu_2}{\partial S_1}(S_1,S_2) < 0,
\quad\displaystyle\frac{\partial \mu_2}{\partial S_2}(S_1,S_2) > 0,
\quad
\forall \; (S_1,S_2) \in \mathbb{R}_+^2\ .$\\
\end{description}
Hypothesis {\bf{A2}} expresses that no growth can take place for species $X_1$ without the substrate $S_1$
and that the  intermediate product $S_2$ is obligate for the growth of species $X_2$.
Hypothesis {\bf{A3}} means that the growth of species $X_1$ increases with the substrate $S_1$
and it is inhibited by the intermediate product $S_2$ that it produces.
Hypothesis {\bf{A4}} means that the growth of species $X_2$ increases with intermediate product $S_2$
produced by species $X_1$ while it is inhibited by the substrate $S_1$.
Note that there is a syntrophic relationship between  the two species.\\

We first scale system (\ref{model0}) using the following change of variables and notations :
$$
\displaystyle s_1 = \frac{k_1}{k_3} S_1,
\quad x_1 = k_1 X_1,
\quad s_2 = S_2,
\quad x_2 = k_2 X_2,
\quad s_1^{in} = \frac{k_1}{k_3} S_1^{in},
\quad s_2^{in}=S_2^{in}\ .
$$
The dimensionless equations thus obtained are :

\begin{equation}
\label{model}
\left\{
\begin{array}{rcl}
\dot{s}_1 & = & D(s_1^{in} - s_1) - f_1(s_1,s_2) x_1\ ,\\
&&\\
\dot{x}_1 & = & f_1(s_1,s_2) x_1 - D x_1\ ,\\
&&\\
\dot{s}_2 & = & D(s_2^{in} - s_2)- f_2(s_1,s_2) x_2+f_1(s_1,s_2) x_1 \ ,\\
&&\\
\dot{x}_2 & = & f_2(s_1,s_2) x_2 - D x_2\ .\\
\end{array}
\right.
\end{equation}
Where the functions $\displaystyle f_1,f_2: \mathbb{R}_+^2 \rightarrow \mathbb{R}_+ $ are defined by
$$ f_1(s_1,s_2) = \mu_1(\frac{k_3}{k_1}s_1,s_2)\quad \mbox{and}\quad f_2(s_1,s_2) = \mu_2(\frac{k_3}{k_1}s_1,s_2).$$
Hypotheses {\bf A1}--{\bf A4} satisfied by the functions $\mu_1$ and $\mu_2$ translate in the following assumptions of the functions $f_1$ and $f_2$:

\begin{description}
\item[H1] $f_1,f_2: \mathbb{R}_+^2 \rightarrow \mathbb{R}_+$, of class $\mathcal{C}^1\ ,$

\item[H2] $f_1(0,s_2) = 0, \qquad f_2(s_1,0) = 0,\quad \forall \;
(s_1,s_2) \in \mathbb{R}_+^2\ ,$

\item[H3] $\displaystyle \frac{\partial f_1}{\partial s_1}(s_1,s_2) > 0,
\quad\displaystyle \frac{\partial f_1}{\partial s_2}(s_1,s_2) < 0,\quad \forall \;
(s_1,s_2) \in \mathbb{R}_+^2\ ,$

\item[H4] $\displaystyle \frac{\partial f_2}{\partial s_1}(s_1,s_2) < 0,
\quad\displaystyle \frac{\partial f_2}{\partial s_2}(s_1,s_2) > 0,\quad \forall \;
(s_1,s_2) \in \mathbb{R}_+^2\ .$
\end{description}
$\mathbb{R}_+^4$, the closed non-negative cone in $\mathbb{R}^4$,
is positively invariant under the solution map of system (\ref{model}).
More precisely
\begin{Prop}
For every initial condition in $\mathbb{R}_+^4$,
the solution of system (\ref{model}) has positive components
and is positively bounded and thus is defined for every positive $t$.
The set
$$\Omega=\Big\{(s_1,x_1,s_2,x_2)\in \mathbb{R}_+^4: \; s_1+x_1=s_1^{in}, \quad s_2+x_2=x_1+s_2^{in}\Big\}$$
is a positive invariant  attractor of all solutions of system (\ref{model}).
\end{Prop}
\begin{proof}
The invariance of $\mathbb{R}_+^4$ is guaranteed by the fact that  :
\begin{itemize}
\item[i.]  $s_1=0 \Rightarrow \dot s_1= D \;s_1^{in}>0$,
\item[ii.] $s_2=0 \Rightarrow \dot s_2=\displaystyle D \;s_2^{in} + f_1(s_1,0) \;x_1> 0$,
\item[iii.] $x_{i}=0 \Rightarrow \dot x_i=0$ for $i=1,2$.
\end{itemize}
Next we have to prove that the solution is bounded.
Let $z_1=s_1+x_1$, then $\dot z_1=-D(z_1-s_1^{in})$ from which one deduces :
\begin{equation}
\label{dz_1/dt}
\displaystyle s_1(t)+x_1(t)=s_1^{in}+(s_1(0)+x_{1}(0)-s_1^{in}) e^{-D t}\ .
\end{equation}
Thus $s_1(t)$ and $x_1(t)$ are positively bounded.
Let $z_2=s_2+x_2-x_1$, then $\dot z_2=-D(z_2-s_2^{in})$
from which one deduces:
\begin{equation}
\label{dz_2/dt}
\displaystyle s_2(t)+x_2(t)-x_1(t)=s_2^{in}+(s_2(0)+x_2(0)-x_1(0)-s_2^{in}) e^{-D t}\ .
\end{equation}
Thus $s_2(t)$ and $x_2(t)$ are positively bounded. Hence, the solution is defined for all positive $t$.
From (\ref{dz_1/dt}) and (\ref{dz_2/dt}) we deduce that the set $\Omega$
is an invariant set which is an attractor.
\end{proof} 
\section{Restriction on the plane}\label{secD2}
The solutions of system (\ref{model}) are exponentially convergent towards the set $\Omega$ and we are interested in the asymptotic behavior of these solutions. It is enough to restrict the study of the asymptotic behaviour of system (\ref{model}) to $\Omega$. In fact, thanks to Thieme's results \cite{thieme}, the asymptotic behaviour of the solutions of the restriction of (\ref{model}) on $\Omega$ will be informative for the complete system, see Section \ref{sec2.3}.
In this section we study the following reduced system which is simply the projection on the plane $(x_1,x_2)$,
of the restriction of system (\ref{model}) on $\Omega$.
\begin{eqnarray}
\left\{
\begin{array}{r}
\displaystyle  \dot x_1 =
\displaystyle \left[\Phi_1(x_1,x_2)- D\right] x_1,\\[2mm]
\displaystyle  \dot x_2 =
\displaystyle \left[\Phi_2(x_1,x_2)- D\right] x_2.
\end{array}
\label{reduit}
\right.
\end{eqnarray}
where
$$
\Phi_1(x_1,x_2)=f_1\left(s_1^{in}-x_1,s_2^{in}+x_1-x_2\right),
\quad
\Phi_2(x_1,x_2)=f_2\left(s_1^{in}-x_1,s_2^{in}+x_1-x_2\right). 
$$
Thus, for (\ref{reduit}) the state-vector $(x_1,x_2)$ belongs to the following subset of the plane, see Fig. \ref{figsetS} :
$$\mathcal{S}=\left\{(x_1,x_2)\in\mathbb{R_+}^2:0< x_1\leq{s_1^{in}},0< x_2\leq x_1+s_2^{in}\right\}.$$
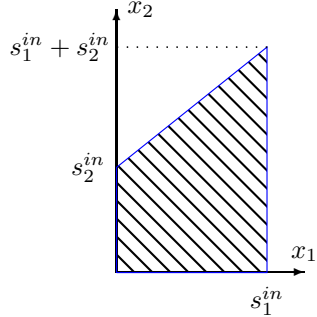
\begin{figure}[!ht]
\begin{center}
\psset{unit=0.5cm}
\begin{pspicture}(-1,-1)(5,7)
   \put(0,0){\vector(1,0){5}}
   \put(0,0){\vector(0,1){7}}
   \psline[fillstyle=vlines,linewidth=0.15mm,linecolor=blue](0,0)(4,0)(4,6)(0,2.8)(0,0)
   \psline[linestyle=dotted,linewidth=0.25mm,linecolor=black](0,6)(4,6)
   \rput(5,0.5){$x_1$}
   \rput(0.65,7){$x_2$}
   \rput(-0.75,2.8){$s_2^{in}$}
   \rput(4,-0.75){${s_1^{in}}$}
   \rput(-1.5,6){${s_1^{in}}+s_2^{in}$}
\end{pspicture}
\caption{The set $\mathcal{S}$\label{figsetS}}
\end{center}
\end{figure} 

The point $F^0=(0,0)$ is an equilibrium of (\ref{reduit}). Besides this equilibrium point the system can have 
the following three types of equilibrium points.
\begin{itemize}
\item  Boundary equilibria
$F^1=(\bar x_1,0)$,
where $x_1=\bar x_1$ is a solution, if it exists, of equation 
\begin{equation}\label{eqx1bar}
\Phi_1(x_1,0)=D,
\end{equation}
\item Boundary equilibria
$F^2=(0,\tilde{x}_2)$,
where $x_2=\tilde{x}_2$ is a solution, if it exists, of equation
\begin{equation}\label{eqx2tilde}
\Phi_2(0,x_2)=D,
\end{equation}
\item Positive equilibria
$F^*=(x_1^*,x_2^*)$,
where $x_1=x_1^*$, $x_2=x_2^*$ is a solution, if it exists, of the system of equations
\begin{equation}\label{eqx1x2}
\left\{
\begin{array}{l}
\Phi_1(x_1,x_2)=D\\
\Phi_2(x_1,x_2)=D.
\end{array}
\right.
\end{equation}
\end{itemize}
We use the following notations 
$$D_1=f_1(s_1^{in},s_2^{in}),\qquad D_2=f_2(s_1^{in},s_2^{in}).$$
The mapping $x_1\mapsto\Phi_1(x_1,0)$ is decreasing, and the mapping $x_1\mapsto\Phi_2(x_1,0)$ is increasing.
If $D_1>D_2$, there exists a unique real number $\xi_1$ satisfying $\Phi_1(\xi_1,0)=\Phi_2(\xi_1,0)$, 
since 
$$\Phi_1(0,0)=D_1>D_2=\Phi_2(0,0),\mbox{ and }\Phi_1(s_1^{in},0)=0<\Phi_2(s_1^{in},0).$$
We denote by $D_3\in]D_2,D_1[$
the unique real number 
(see Figure \ref{figcas1}, right) such that:
$$\Phi_1(\xi_1,0)=\Phi_2(\xi_1,0)=D_3.$$
\begin{figure}[ht]
\setlength{\unitlength}{1.0cm}
\begin{center}
\begin{picture}(8.5,4.5)(0,0)
\put(0,0.5){\includegraphics[scale=0.2]{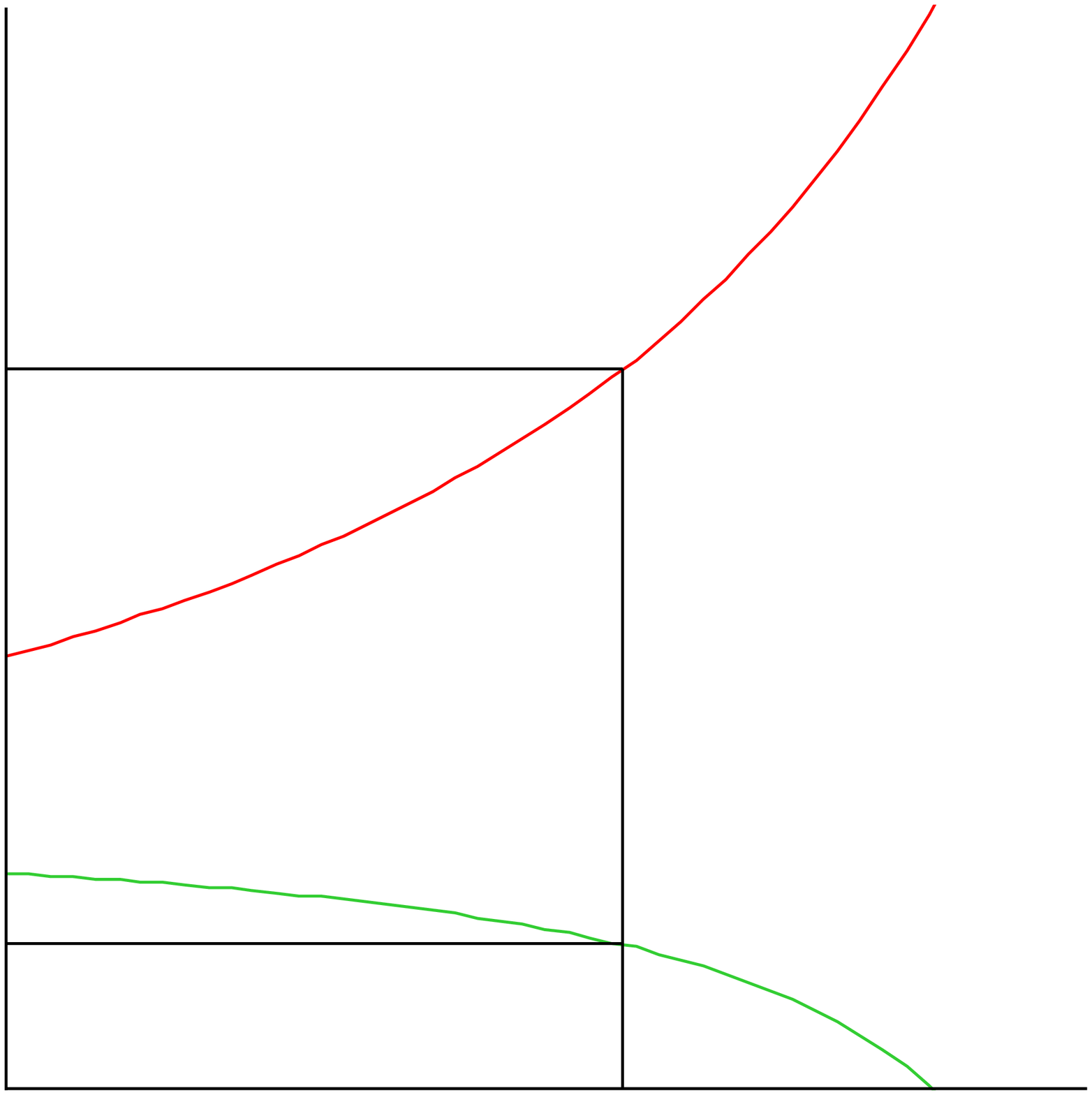}}
\put(5.5,0.5){\includegraphics[scale=0.2]{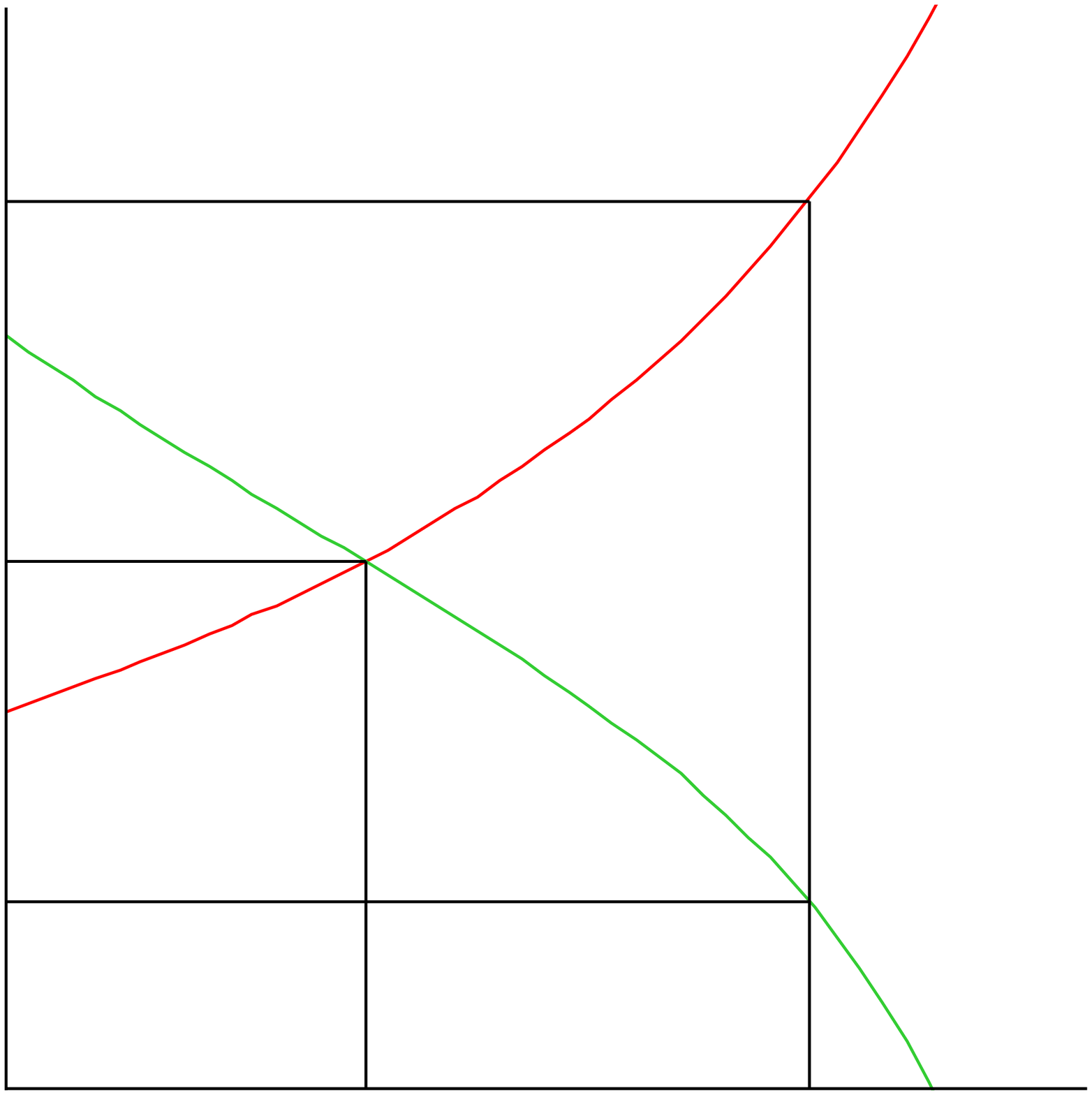}}
\put(-0.15,4.3){{\small$y$}}
\put(1.2,4){{\small$y=\Phi_2(x_1,0)$}}
\put(2.7,1){{\small$y=\Phi_1(x_1,0)$}}
\put(-0.15,0.3){{\small$0$}}
\put(-0.4,2){{\small$D_2$}}
\put(-0.4,1.3){{\small$D_1$}}
\put(-0.25,0.95){{\small$D$}}
\put(-1.25,3){{\small$\Phi_2(\bar x_1,0)$}}
\put(3.8,0.35){{\small$x_1$}}
\put(2.15,0.3){{\small$\bar x_1$}}
\put(3.2,0.25){{\small$s_1^{in}$}}
\put(5.35,4.3){{\small$y$}}
\put(8.7,3.9){{\small$y=\Phi_2(x_1,0)$}}
\put(8.7,0.9){{\small$y=\Phi_1(x_1,0)$}}
\put(5.4,0.3){{\small$0$}}
\put(5.15,1.85){{\small$D_2$}}
\put(5.15,3.15){{\small$D_1$}}
\put(5.15,2.35){{\small$D_3$}}
\put(5.25,1.15){{\small$D$}}
\put(4.25,3.6){{\small$\Phi_2(\bar x_1,0)$}}
\put(9.4,0.35){{\small$x_1$}}
\put(8.3,0.3){{\small$\bar x_1$}}
\put(8.7,0.25){{\small$s_1^{in}$}}
\put(6.7,0.3){{\small$\xi_1$}}
\end{picture}
\end{center}
\caption{Existence and uniqueness of $\bar x_1$. 
On the left, the case $D_1<D_2$: $\Phi_2(\bar x_1,0)>D$ for all $D<D_1$.
On the right, the case $D_1>D_2$: $\Phi_2(\bar x_1,0)>D$ if and only if $D<D_3$.} 
\label{figcas1}
\end{figure} 

The mapping $x_2\mapsto\Phi_1(0,x_2)$ is increasing, and the mapping $x_2\mapsto\Phi_2(0,x_2)$ is decreasing.
Hence, if $D_1<D_2$, there exists a unique real number $\xi_2$ satisfying $\Phi_1(0,\xi_2)=\Phi_2(0,\xi_2)$, 
since 
$$\Phi_2(0,0)=D_2>D_1=\Phi_1(0,0),\mbox{ and }\Phi_2(0,s_2^{in})=0<\Phi_1(0,s_2^{in}).$$
We denote by $D_4\in]D_1,D_2[$
the unique real number 
(see Figure \ref{figcas2}, right) such that:
$$\Phi_1(0,\xi_2)=\Phi_2(0,\xi_2)=D_4.$$
\begin{figure}[ht]
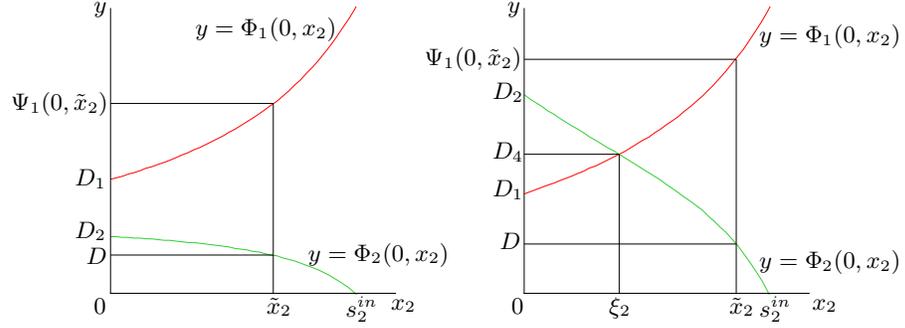

\setlength{\unitlength}{1.0cm}
\begin{center}
\begin{picture}(8.5,4.5)(0,0)
\put(0,0.5){\includegraphics[scale=0.2]{cas1.eps}}
\put(5.5,0.5){\includegraphics[scale=0.2]{cas2.eps}}
\put(-0.15,4.3){{\small$y$}}
\put(1.2,4){{\small$y=\Phi_1(0,x_2)$}}
\put(2.7,1){{\small$y=\Phi_2(0,x_2)$}}
\put(-0.15,0.3){{\small$0$}}
\put(-0.4,2){{\small$D_1$}}
\put(-0.4,1.3){{\small$D_2$}}
\put(-0.25,0.95){{\small$D$}}
\put(-1.25,3){{\small$\Psi_1(0,\tilde x_2)$}}
\put(3.8,0.35){{\small$x_2$}}
\put(2.15,0.3){{\small$\tilde x_2$}}
\put(3.2,0.25){{\small$s_2^{in}$}}
\put(5.35,4.3){{\small$y$}}
\put(8.7,3.9){{\small$y=\Phi_1(0,x_2)$}}
\put(8.7,0.9){{\small$y=\Phi_2(0,x_2)$}}
\put(5.4,0.3){{\small$0$}}
\put(5.15,1.85){{\small$D_1$}}
\put(5.15,3.15){{\small$D_2$}}
\put(5.15,2.35){{\small$D_4$}}
\put(5.25,1.15){{\small$D$}}
\put(4.25,3.6){{\small$\Psi_1(0,\tilde x_2)$}}
\put(9.4,0.35){{\small$x_2$}}
\put(8.3,0.3){{\small$\tilde x_2$}}
\put(8.7,0.25){{\small$s_2^{in}$}}
\put(6.7,0.3){{\small$\xi_2$}}
\end{picture}
\end{center}
\caption{Existence and uniqueness of $\tilde x_2$. 
On the left, the case $D_2<D_1$: $\Phi_1(0,\tilde x_2)>D$ for all $D<D_2$.
On the right, the case $D_2>D_1$: $\Phi_1(0,\tilde x_2)>D$ if and only if $D<D_4$.} 
\label{figcas2}
\end{figure} 
The nature of the trivial equilibrium point $F^0$ is given in the following lemma.
\begin{Lem}
If $D>\max(D_1,D_2)$ then
$F^0$ is a stable node. If $\min(D_1,D_2)<D<\max(D_1,D_2)$
then $F^0$ is a saddle point. If $D<\min(D_1,D_2)$ then
$F^0$ is an unstable node.
\end{Lem}
\begin{proof}
The Jacobian matrix $J$ of (\ref{reduit}), at point $(x_1,x_2)$, is given by:
\begin{eqnarray*}
J=\left[
\begin{array}{ccc}
\displaystyle -\frac{\partial f_1}{\partial s_1} x_1+ \frac{\partial f_1}{\partial s_2} x_1+f_1-D
&~&
\displaystyle -\frac{\partial f_1}{\partial s_2} x_1\\
&&\\
\displaystyle -\frac{\partial f_2}{\partial s_1} x_2+ \frac{\partial f_2}{\partial s_2} x_2
& & -\displaystyle  \frac{\partial f_2}{\partial s_2} x_2+f_2-D
\end{array}
\right].
\end{eqnarray*}
where the functions are evaluated at $\left(s_1^{in}-x_1,s_2^{in}+x_1-x_2\right)$.
The Jacobian matrix at $F^0$ is given by:
\begin{eqnarray*}
J^0=\left[
\begin{array}{cc}
\displaystyle f_1(s_1^{in},s_2^{in})-D & \displaystyle  0\\
&\\
\displaystyle 0&\displaystyle f_2(s_1^{in},s_2^{in})-D
\end{array}
\right]
\end{eqnarray*}
The eigenvalues are $D_1-D$ and $D_2-D$. Thus, if
$D>\max(D_1,D_2)$ then  $F^0$ is a stable node.
It is an unstable node if
$D<\min(D_1,D_2)$.
It is a saddle point if
$\min(D_1,D_2)<D<
\max(D_1,D_2)$.
\end{proof}

The conditions of existence of the boundary equilibria
$F^1$ and $F^2$, and their nature, are stated in the following lemmas.

\begin{Lem}\label{lemx1bar}
An equilibrium
$F^1=(\bar x_1,0)$ exists if and only if $D<D_1$. If  it exists
then it the unique equilibrium on the positive $x_1$ semi-axis. 
If $D_1<D_2$ then $F^1$ is a saddle point for all $D<D_1$. If $D_2<D_1$,
then $F^1$ is a saddle point for all $0<D<D_3$ and a stable node
for all $D_3<D<D_1$.
\end{Lem}
\begin{proof}
An equilibrium
$F^1=(\bar x_1,0)$ exists if and only if
$x_1=\bar x_1\in]0,s_1^{in}[$ is a solution of (\ref{eqx1bar}).
Let $\psi_1(x_1)=\Phi_1(x_1,0)$. Then
$$\psi_1'(x_1)=-\frac{\partial f_1}{\partial s_1}(s_1^{in}-x_1,s_2^{in}+ x_1)+
\frac{\partial f_1}{\partial s_2}(s_1^{in}-x_1,s_2^{in}+ x_1).$$
By assumption {\bf H3},
$\psi_1'(x_1)<0$. 
Since
$\psi_1(0)=D_1$, and 
$\psi_1(s_1^{in})=0$,
equation (\ref{eqx1bar}) admits a solution in the interval $]0,s_1^{in}[$ if and only if $D<D_1$. 
If this condition is satisfied then (\ref{eqx1bar}) admits a unique solution 
since the function $\psi_1(.)$ is decreasing,  see Figure \ref{figcas1}.
The Jacobian matrix at $F^1$ is given by:
\begin{eqnarray*}
J^1=\left[
\begin{array}{cc}
\displaystyle -\frac{\partial f_1}{\partial s_1} \bar x_1+ \frac{\partial f_1}{\partial s_2} \bar x_1
&\displaystyle -\frac{\partial f_1}{\partial s_2} \bar x_1\\
&\\
\displaystyle 0&  \displaystyle  f_2-D \\
\end{array}
\right]
\end{eqnarray*}
where the functions are evaluated at $(s_1^{in}-\bar x_1, s_2^{in}+\bar x_1)$.
The eigenvalues are 
$$\displaystyle f_2(s_1^{in}-\bar x_1, s_2^{in}+\bar x_1)-D=\Phi_2(\bar x_1,0)-D,
\mbox{ and }
\displaystyle -\frac{\partial f_1}{\partial s_1} \bar x_1+ \frac{\partial f_1}{\partial s_2}\bar x_1<0.$$
Thus $F^1$ is a saddle point if $\Phi_2(\bar x_1,0)>D$. If $D_1<D_2$, this condition is satisfied for all $D<D_1$. 
If $D_2<D_1$, it is statisfied for all $0<D<D_3$, see Figure \ref{figcas1}.
 $F^1$ is a stable node if $D_3<D<D_1$ and  $D_2<D_1$. 
\end{proof}

\begin{Lem}\label{lemx2tilde}
An equilibrium
$F^2=(0,\tilde x_2)$ exists if and only if $D<D_2$. If  it exists
then it the unique equilibrium on the positive $x_2$ semi-axis. 
If $D_2<D_1$ then $F^2$ is a saddle point for all $D<D_2$. If $D_1<D_2$,
then $F^2$ is a saddle point for all $0<D<D_4$ and a stable node
for all $D_4<D<D_2$.
\end{Lem}
\begin{proof}
An equilibrium
$F^2=(0,\tilde x_2)$ exists if and only if $x_2=\tilde x_2\in]0,s_2^{in}[$ is a solution of (\ref{eqx2tilde}).
Let $\psi_2(x_2)=\Phi_2(0,x_2)$. 
Then
$$\psi_2'(x_2)=-\frac{\partial f_1}{\partial s_2}(s_1^{in},s_2^{in}-\tilde x_2).$$
By assumption {\bf H4},
$\psi_2'(x_2)<0$. 
Since
$\psi_2(0)=D_2$, and 
$\psi_2(s_2^{in})=0$,
equation (\ref{eqx2tilde}) admits a solution in the interval $]0,s_2^{in}[$ if and only if $D<D_2$.
If this condition is satisfied then (\ref{eqx2tilde}) admits a unique solution 
since the function $\psi_2(.)$ is decreasing, see Figure \ref{figcas2}.
The Jacobian matrix at $F^2$ is given by:
\begin{eqnarray*}
J^2=\left[
\begin{array}{cc}
\displaystyle f_1-D&0\\
&\\
\displaystyle -\frac{\partial f_2}{\partial s_1} \tilde x_2+ \frac{\partial f_2}{\partial s_2} \tilde x_2
&  -\displaystyle  \frac{\partial f_2}{\partial s_2} \tilde x_2
\end{array}
\right]
\end{eqnarray*}
where the functions are evaluated at $(s_1^{in},s_2^{in}-\tilde x_2)$.
The eigenvalues are
$$\displaystyle f_1(s_1^{in},s_2^{in}-\tilde x_2)-D=\Phi_1(0,\tilde x_2)-D,
\mbox{ and }
\displaystyle -\frac{\partial f_2}{\partial s_2} \tilde x_2<0.$$
Thus $F^2$ is a saddle point if $\Phi_1(0,\tilde x_2)>D$.
 If $D_2<D_1$, this condition is satisfied for all $D<D_2$. 
If $D_1<D_2$, it is statisfied for all $0<D<D_4$, see Figure \ref{figcas2}.
 $F^2$ is a stable node if $D_4<D<D_2$ and  $D_1<D_2$. 
\end{proof}

Let us discuss now the conditions of existence of positive equilibria $F^*$, and their number.
An equilibrium
$F^*=(x_1^*,x_2^*)$ exists if and only if $x_1=x_1^*$, $x_2=x_2^*$ is a solution of (\ref{eqx1x2}) 
lying in $\mathcal{S}$.
One has
$$\frac{\partial \Phi_1}{\partial x_2}=-\frac{\partial f_1}{\partial s_2}(s_1^{in}-x_1,s_2^{in}+ x_1-x_2).$$
By assumption {\bf H3}, this partial derivative is positive. Hence, equation $\Phi_1(x_1,x_2)=D$ 
defines a function
$x_2=F_1(x_1)$
such that $F_1(\bar x_1)=0$ when $D<D_1$. 
Recall that $x_1=\bar x_1$ is the solution of (\ref{eqx1bar}) which, according to 
Lemma \ref{lemx1bar} exists and is unique, if and only if $D<D_1$.
One has
$$F'_1(x_1)=
-\frac{\frac{\partial \Phi_1}{\partial x_1}(x_1,F_1(x_1))}{\frac{\partial \Phi_1}{\partial x_2}(x_1,F_1(x_1))}=
\frac{-\frac{\partial f_1}{\partial s_1}+\frac{\partial f_1}{\partial s_2}}{\frac{\partial f_1}{\partial s_2}}
=1-\frac{\frac{\partial f_1}{\partial s_1}}{\frac{\partial f_1}{\partial s_2}}>1.
$$
Hence the function $F_1$ is increasing. 
Since $\Phi_1(s_1^{in},0)=0$, the graph $\Gamma_1$ of $F_1$
has no intersection with the right boundary of the domain $\mathcal{S}$, defined by $x_1=s_1^{in}$.
This graph separates $\mathcal{S}$ in two regions denoted as the left and right sides of $\Gamma_1$, 
see Figure \ref{figLR}.
One has also
$$\frac{\partial \Phi_2}{\partial x_2}=-\frac{\partial f_2}{\partial s_2}(s_1^{in}-x_1,s_2^{in}+ x_1-x_2).$$
By assumption {\bf H3}, this partial derivative is positive. Hence, equation $\Phi_2(x_1,x_2)=D$ defines a function
$x_2=F_2(x_1)$
such that $F_2(0)=\tilde x_2$ when $D<D_2$. Recall that $x_2=\tilde x_2$ is the solution of (\ref{eqx2tilde}) 
which, according to Lemma \ref{lemx2tilde} exists and is unique, if and only if $D<D_2$.
One has
$$F'_2(x_1)=
-\frac{\frac{\partial \Phi_2}{\partial x_1}(x_1,F_2(x_1))}{\frac{\partial \Phi_2}{\partial x_2}(x_1,F_2(x_1))}=
\frac{-\frac{\partial f_2}{\partial s_1}+\frac{\partial f_2}{\partial s_2}}{\frac{\partial f_2}{\partial s_2}}
=1-\frac{\frac{\partial f_2}{\partial s_1}}{\frac{\partial f_2}{\partial s_2}}>1.
$$
Hence the function $F_2$ is increasing. 
Since $\Phi_2(x_1,s_2^{in}+x_1)=0$, the graph $\Gamma_2$ of $F_2$ 
has no intersection with the top boundary of the domain $\mathcal{S}$, 
defined by $x_2=s_2^{in}+x_1$.
Thus the point at the very right of $\Gamma_2$ lies necessarily on the right boundary of 
$\mathcal{S}$, defined by $x_1=s_1^{in}$. Hence it lies on the right side of $\Gamma_1$, see Figure \ref{figLR}.
\begin{figure}[ht]
\setlength{\unitlength}{1.0cm}
\begin{center}
\begin{picture}(9,4.5)(3.5,0.5)
\put(3,0.5){\includegraphics[scale=0.2]{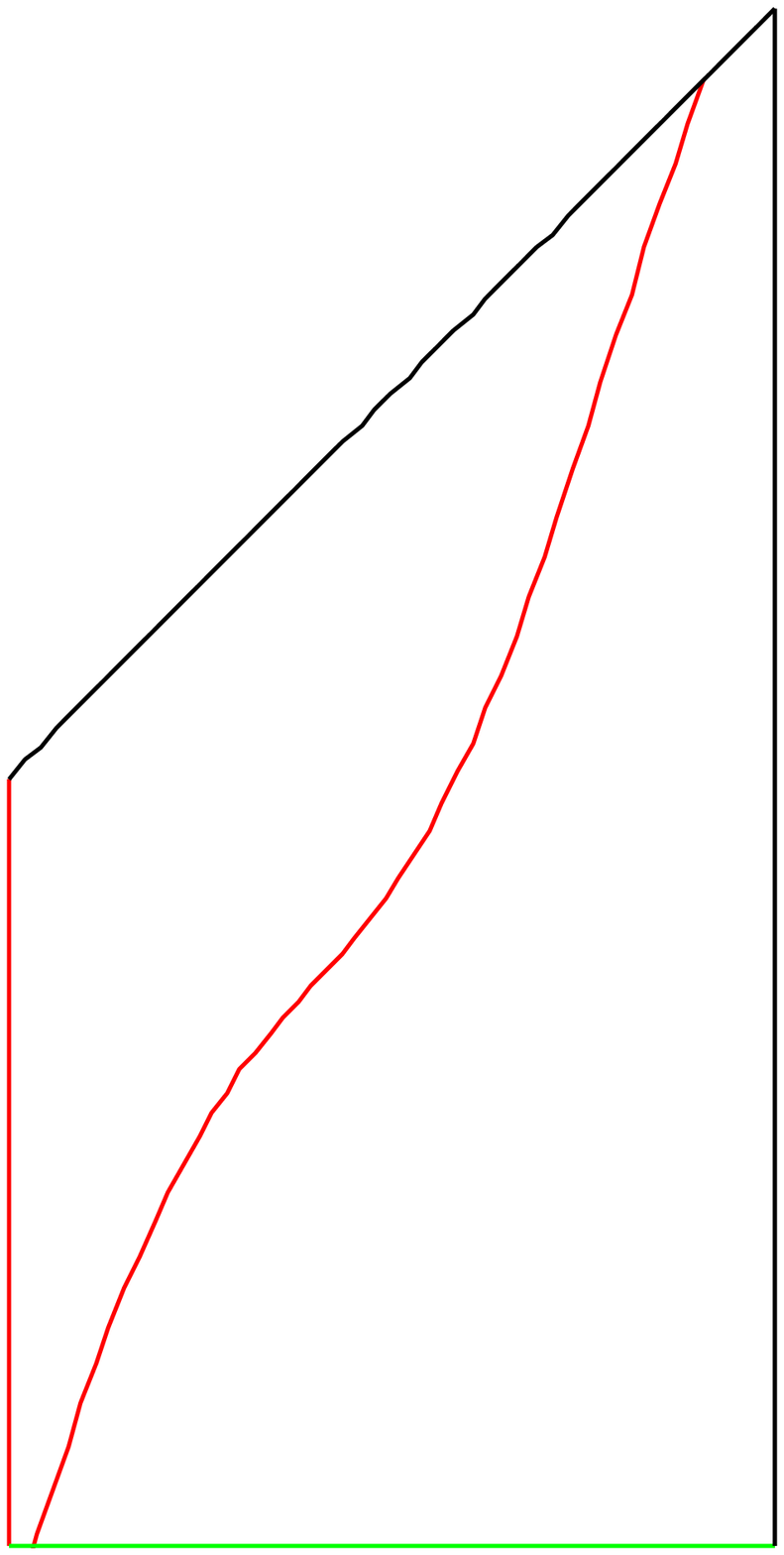}}
\put(6,0.5){\includegraphics[scale=0.2]{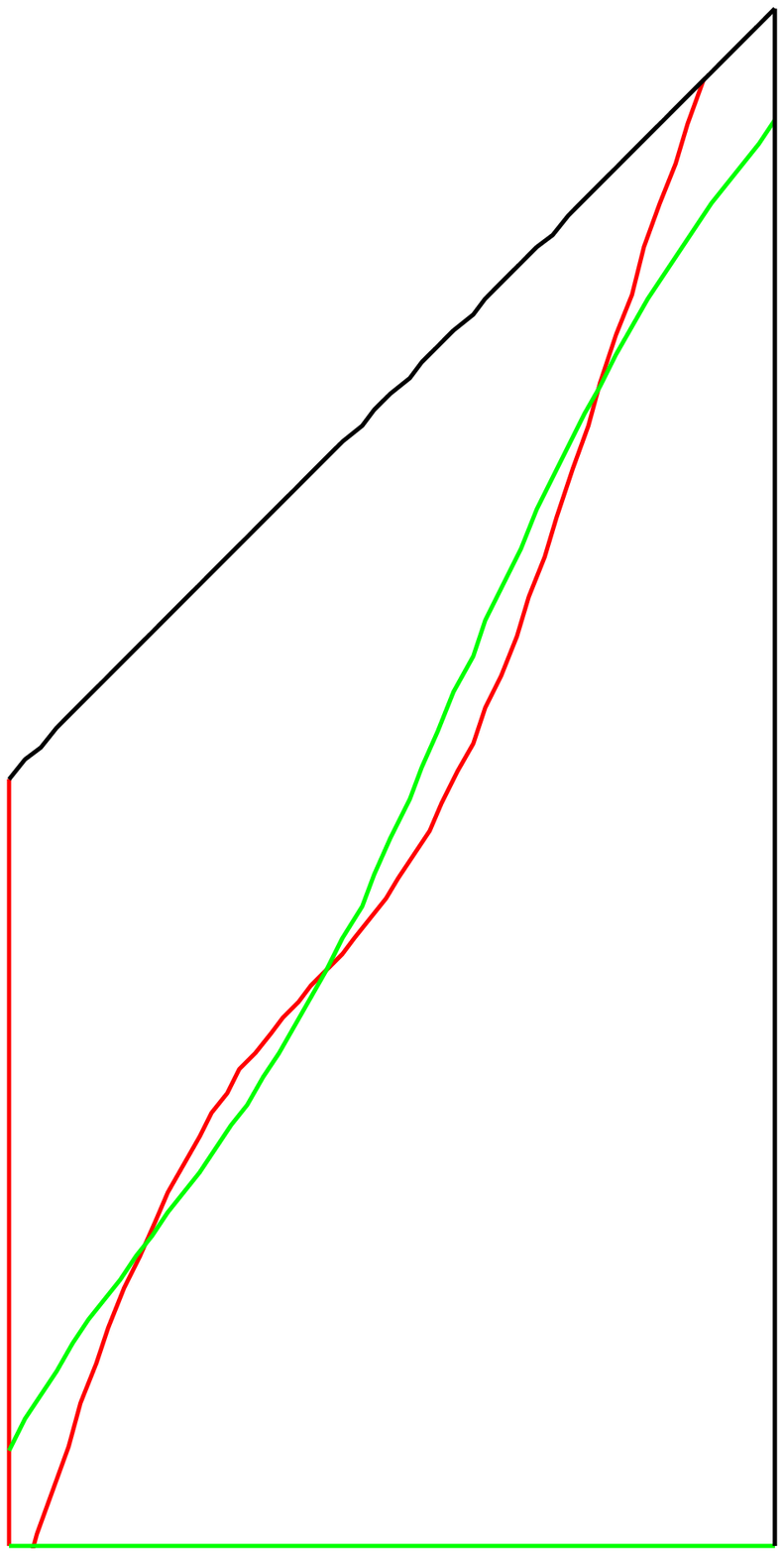}}
\put(9,0.5){\includegraphics[scale=0.2]{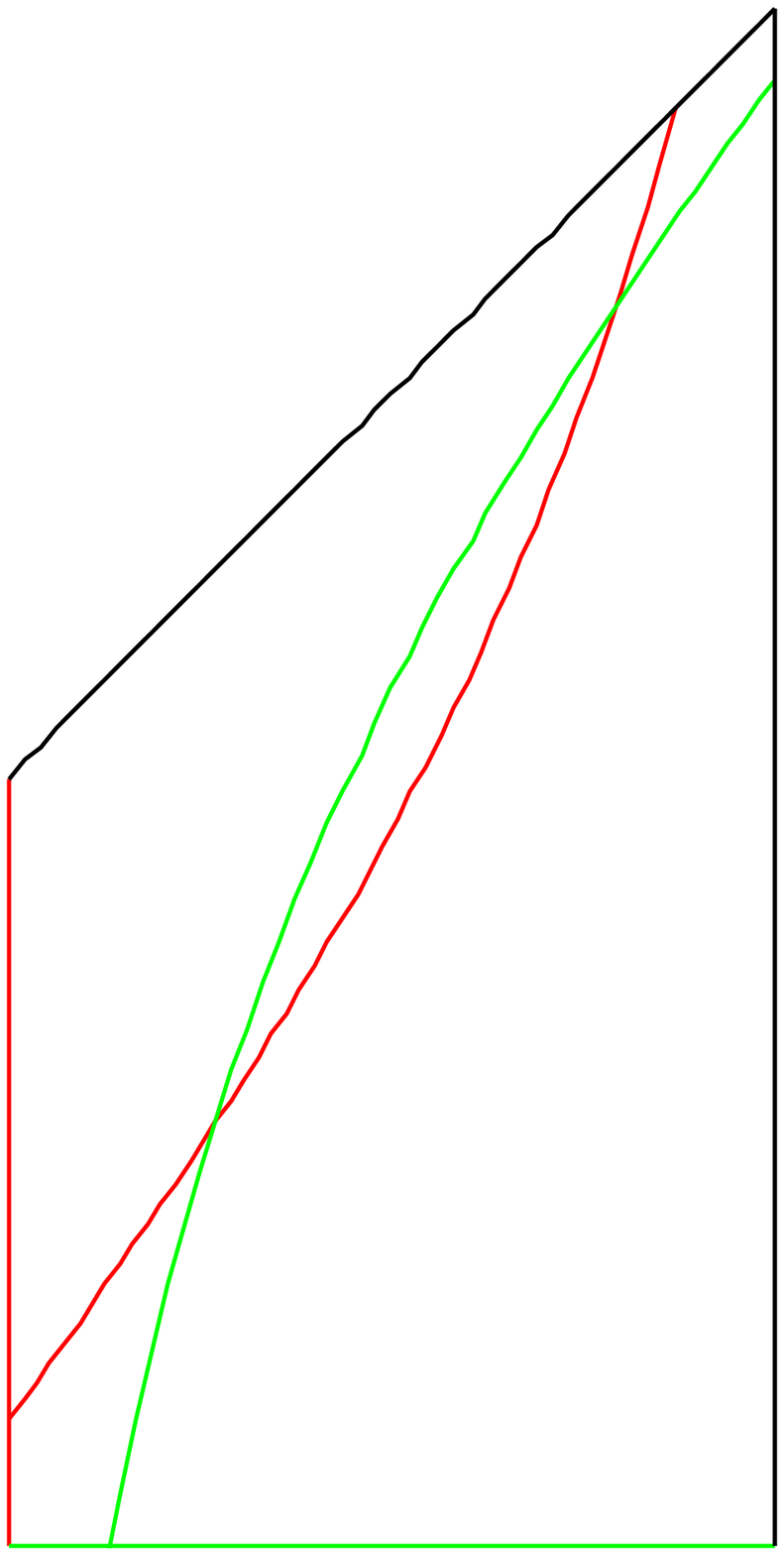}}
\put(4.2,2.2){{\small Left}}
\put(4.6,1){{\small Right}}
\put(5,2){{\small$\Gamma_1$}}
\put(8.42,3.41){$\bullet$}
\put(7.25,1.16){$\bullet$}
\put(7.7,1.91){$\bullet$}
\put(7.9,3){{\small$\Gamma_2$}}
\put(8.2,2.25){{\small$\Gamma_1$}}
\put(7.3,0.95){{\small$F_1^*$}}
\put(7.8,1.7){{\small$F_2^*$}}
\put(8.5,3.2){{\small$F_3^*$}}
\put(6.7,0.8){{\small$A$}}
\put(10.1,0.3){{\small$A$}}
\put(10.7,2.9){{\small$\Gamma_2$}}
\put(11.5,2.2){{\small$\Gamma_1$}}
\put(10.43,1.5){$\bullet$}
\put(11.44,3.55){$\bullet$}
\put(10.5,1.25){{\small$F_1^*$}}
\put(11.5,3.3){{\small$F_2^*$}}
\end{picture}
\end{center}
\caption{On the left, the left and right sides of $\Gamma_1$. On the center, 
the point $A$ at the very left of $\Gamma_2$ lies on left side of $\Gamma_1$: 
there are generically an odd number of intersections (3 in this example). 
On the right, the point $A$ at the very left of $\Gamma_2$ lies on right side 
of $\Gamma_1$: there are generically an even number of intersections (2 in this example).} 
\label{figLR}
\end{figure} 

The graphs $\Gamma_1$ and $\Gamma_2$ can intersect or not, see Figures \ref{figLR}, \ref{figisoclines} and \ref{figbifurcations}. 
If they intersect at some point $F^*=(x_1^*,x_2^*)$ then $F^*$ is a positive equilibbrium.
If the point $A$ at the very left of $\Gamma_2$ lies on left side of $\Gamma_1$ 
then $\Gamma_1$ and $\Gamma_2$ intersect in at least one point $F^*=(x_1^*,x_2^*)$. 
They can have multiple intersections. 
Generically they have an odd number of intersections (see Figure \ref{figLR}, center). 
If the point $A$ at the very left of $\Gamma_2$ lies on right side of $\Gamma_1$ then $\Gamma_1$ and $\Gamma_2$ can intersect
or not.  Generically they have an even number of intersections (see Figure \ref{figLR}, right).
The nature of a positive equilibrium
$F^*$ is stated in the following lemmas.
\begin{Lem}
If an equilibrium $F^*=(x_1^*,x_2^*)$ exists then it is a stable node if $F_1'(x_1^*)>F_2'(x_1^*)$.
It is a saddle point if the opposite inequality is satisfied.
\end{Lem}
\begin{proof}
The Jacobian matrix at $F^*$ is given by:
\begin{eqnarray*}
J^*=\left[
\begin{array}{ccc}
\displaystyle -\frac{\partial f_1}{\partial s_1} x^*_1+ \frac{\partial f_1}{\partial s_2} x^*_1&
&\displaystyle -\frac{\partial f_1}{\partial s_2} x^*_1\\
&&\\
\displaystyle -\frac{\partial f_2}{\partial s_1} x^*_2+ \frac{\partial f_2}{\partial s_2} x^*_2& 
& -\displaystyle  \frac{\partial f_2}{\partial s_2} x^*_2\\
\end{array}
\right]
\end{eqnarray*}
where the derivatives are evaluated at $(s_1^{in}-x_1^*,s_2^{in}+x_1^*-x_2^*)$.
Notice that
$${\rm tr}\displaystyle (J^*)=-\frac{\partial f_1}{\partial s_1} x^*_1+ 
\frac{\partial f_1}{\partial s_2} x^*_1-\displaystyle  \frac{\partial f_2}{\partial s_2} x^*_2<0$$
and
$$\det(J^*)=x^*_1x^*_2\left[
\frac{\partial f_1}{\partial s_1}\frac{\partial f_2}{\partial s_2}-
\frac{\partial f_1}{\partial s_2}\frac{\partial f_2}{\partial s_1}
\right]=
x^*_1x^*_2
\frac{\partial f_1}{\partial s_2}\frac{\partial f_2}{\partial s_2}
\left[F'_2(x_1^*)-F'_1(x_1^*)\right].
$$
By Assumptions {\bf H3} and {\bf H4}, the product of the partial derivatives is negative.
Therefore, the determinant is positive if $F'_1(x_1^*)>F'_2(x_1^*)$ and negative 
if the opposite inequality is satisfied. Hence the equilibrium $F^*=(x_1^*,x_2^*)$ is a stable node 
if $F_1'(x_1^*)>F_2'(x_1^*)$. 
It is a saddle point if the opposite inequality is satisfied.
\end{proof}

The number of equilibria of (\ref{reduit}) and their nature are summarized in the next theorem.

\begin{Th}
\begin{enumerate}
\item If $D<\min(D_1,D_2)$  then
(\ref{reduit}) admits  the trivial equilibrium
$F^0$ which is an unstable node, the boundary equilibria
$F^1$ and $F^2$ which are saddle points, and at least one positive equilibrium
$F^*$. If $F^*$ is the unique positive equilibrium then it is a stable node. Generically, the system
has an odd number of positive equilibria which are alternatively stable nodes and saddle points, 
the one at the very left of these positive equilibria is a stable node.
\item{If $\min(D_1,D_2)<D<\max(D_1,D_2)$, four subcases must be distinguished
\begin{enumerate}
 \item If $D_1<D_2$ and $D_1<D<D_4$ then
(\ref{reduit}) admits the trivial and boundary equilibria
$F^0$  and $F^2$, which are saddle points and at least one positive equilibrium
$F^*$. If $F^*$ is the unique positive equilibrium then it is a stable node. 
Generically, the system has an odd number of positive equilibria which are alternatively 
stable nodes and saddle points, the one at the very left of these
positive equilibria is a stable node.
 \item If $D_1<D_2$ and $D_4<D<D_2$ then
(\ref{reduit}) admits the trivial equilibrium 
$F^0$, which is a saddle point, and the boundary equilibrium
$F^2$, which is a stable node. Generically, the system
can have an even number of positive equilibria which are alternatively 
saddle points and stable nodes, the one at the very left of these
positive equilibria is a saddle point.
\item If $D_2<D_1$ and $D_2<D<D_3$  then
(\ref{reduit}) admits the trivial and boundary equilibria 
$F^0$  and $F^1$ which are saddle points
and at least one positive equilibrium
$F^*$. If $F^*$ is the unique positive equilibrium then it is a stable node. 
Generically, the system has an odd number of positive equilibria which are alternatively 
stable nodes and saddle points, the one at the very left of these
positive equilibria is a stable node.
\item If $D_2<D_1$ and $D_3<D<D_1$ then
(\ref{reduit}) admits the trivial equilibrium 
$F^0$, which is a saddle point, and the boundary equilibrium
$F^1$, which is a stable node. Generically, the system
can have an even number of positive equilibria which are alternatively 
saddle points and stable nodes, the one at the very left of these
positive equilibria is a saddle point.
\end{enumerate}}
\item If $D>\max(D_1,D_2)$ then
(\ref{reduit}) admits the trivial equilibrium 
$F^0$ which is a stable node.
Generically, the system
can have an even number of positive equilibria which are alternatively 
saddle points and stable nodes, the one at the very left of these
positive equilibria is a saddle point.
\end{enumerate}
\label{th2}
\end{Th}
\section{Growth functions of Monod type}\label{sec3}
In this section we consider growth functions $f_1$ and $f_2$ of the following form
\begin{equation}\label{monod}
f_1(s_1,s_2)=\frac{m_1s_1}{(K_1+s_1)(L_1+s_2)},
\qquad
f_2(s_1,s_2)=\frac{m_2s_1}{(K_2+s_1)(L_2+s_2)}.
\end{equation}
Such functions are simply the product of a Monod function in $s_1$ by a
decreasing functions of $s_2$.
Such functions are currently used in biotechnology when the growth of a
functional species is
limited by a substrate while inhibited by another one. Such situations
are common in water treatment technology
like in the denitrification (limited by the nitrate and inhibited by the
dissolved oxygen) or in the anoxic or anaerobic
hydrolysis (limited by the slowly biodegradable substrates while
inhibited by an excess of oxygen) processes
which are modeled this way (cf. \cite{iwa}). 

One can readily check that (\ref{monod})
satisfy Assumptions {\bf{H1}} to {\bf{H4}}. By straighforward calculations one has
$$
F_1(x_1)=
\frac{-Dx_1^2+\left[m_1+D\left(K_1-L_1+s_1^{in}-s_2^{in}\right)\right]-m_1s_1^{in}+
D\left(K_1+s_1^{in}\right)\left(L_1+s_2^{in}\right)}{D(K_1+s_1^{in}-x_1)}
$$
$$
F_2(x_1)=
\frac{Dx_1^2+\left[m_2+D\left(L_2-K_2+s_2^{in}-s_1^{in}\right)\right]+m_2s_2^{in}-
D\left(K_2+s_1^{in}\right)\left(L_2+s_2^{in}\right)}{m_2-D(K_2+s_1^{in}-x_1)}
$$
Hence equation $F_1(x_1)=F_2(x_1)$ giving the abscissa of positive equilibria 
is an algebraic equation of degree 2. Thus, it cannot have more than two solutions. 
Hence, the situation depicted on the center of Figure \ref{figLR}, of three positive equilibria, 
is excluded. However, the situation depicted in the right of Figure \ref{figLR}, with two 
positive equilibria can occur.
\begin{figure}[ht]
\setlength{\unitlength}{1.0cm}
\begin{center}
\begin{picture}(15,4.5)(0.5,0)
\put(0,0.5){\includegraphics[scale=0.2]{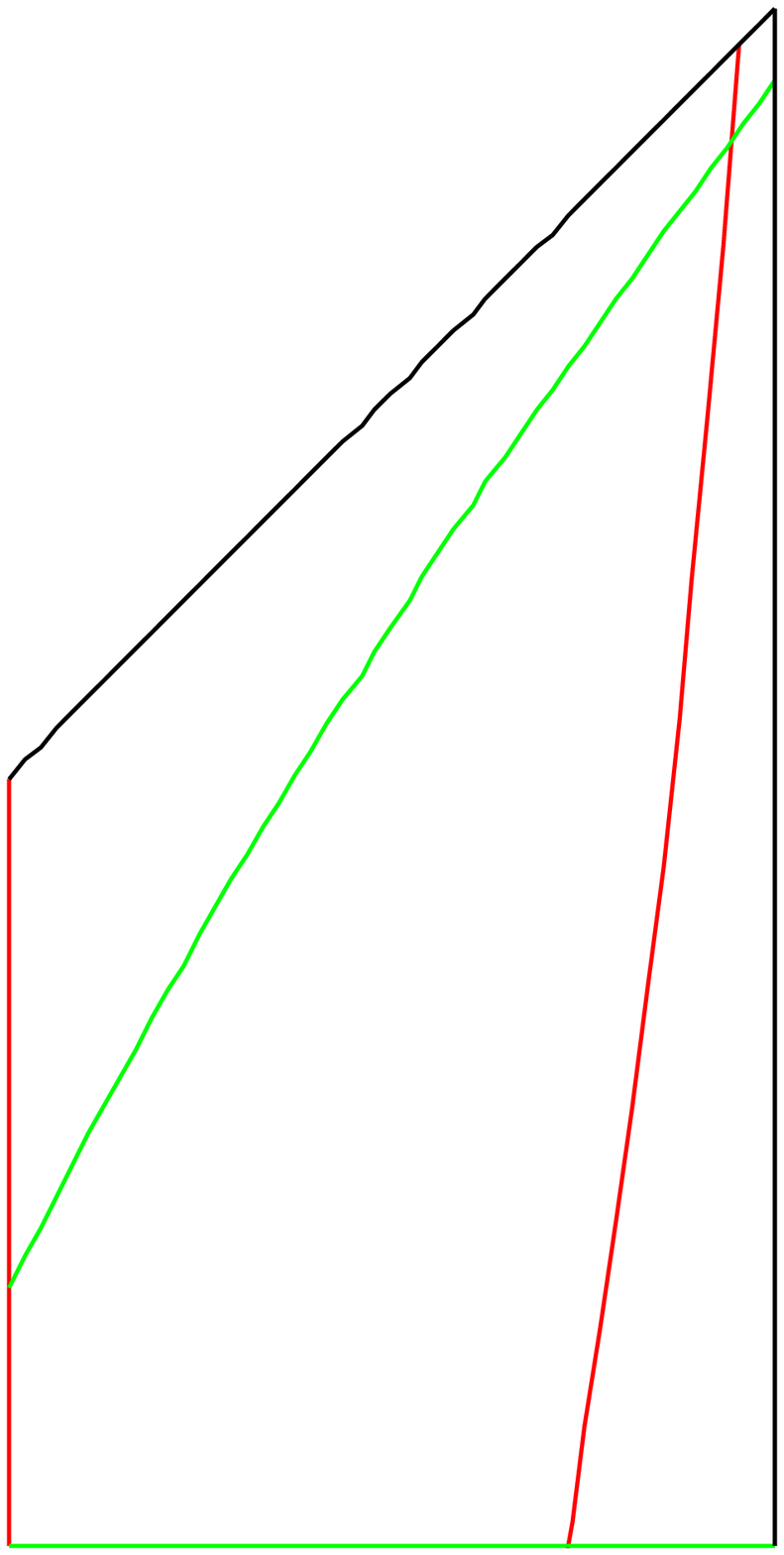}}
\put(3,0.5){\includegraphics[scale=0.2]{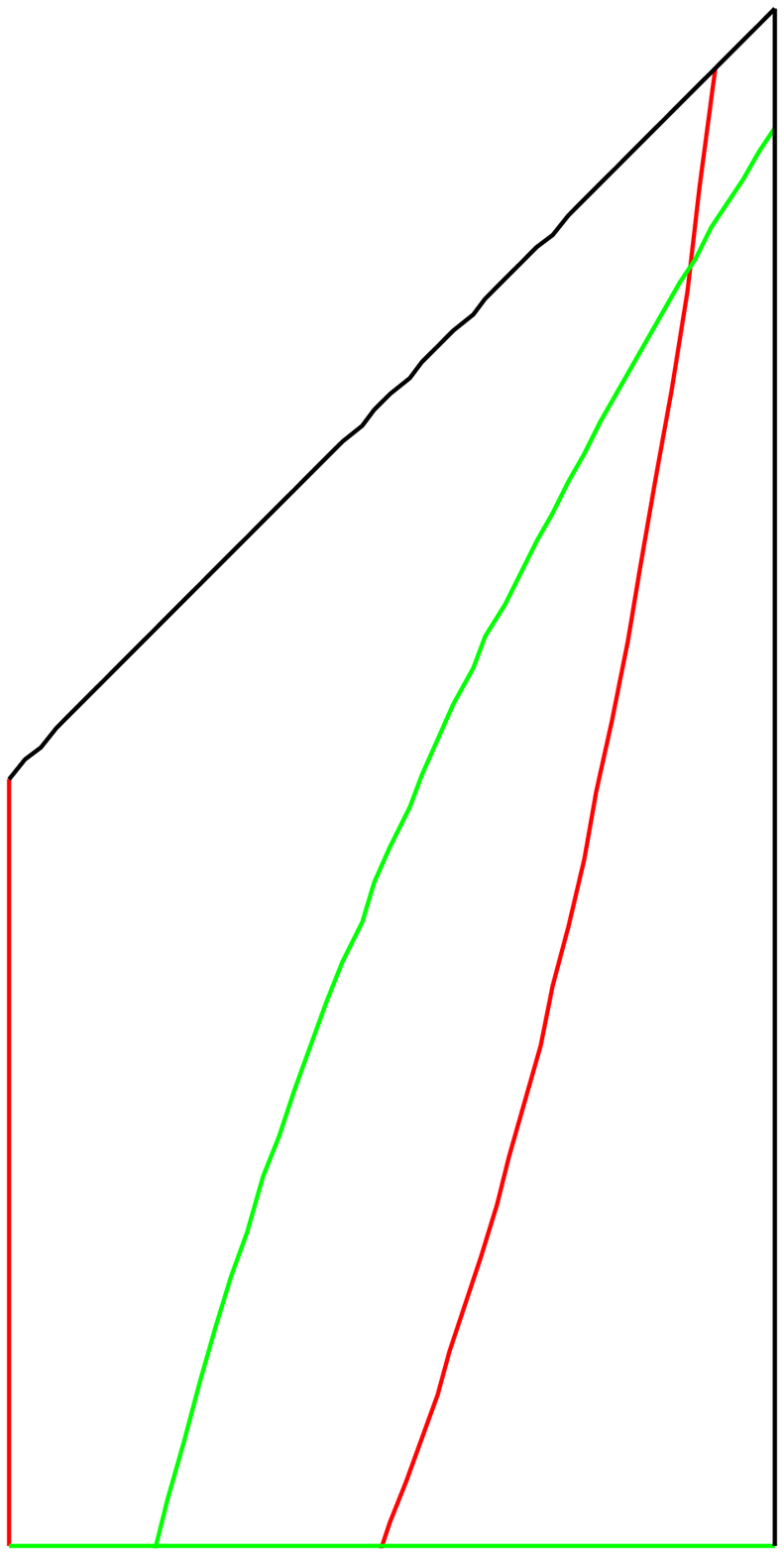}}
\put(6,0.5){\includegraphics[scale=0.2]{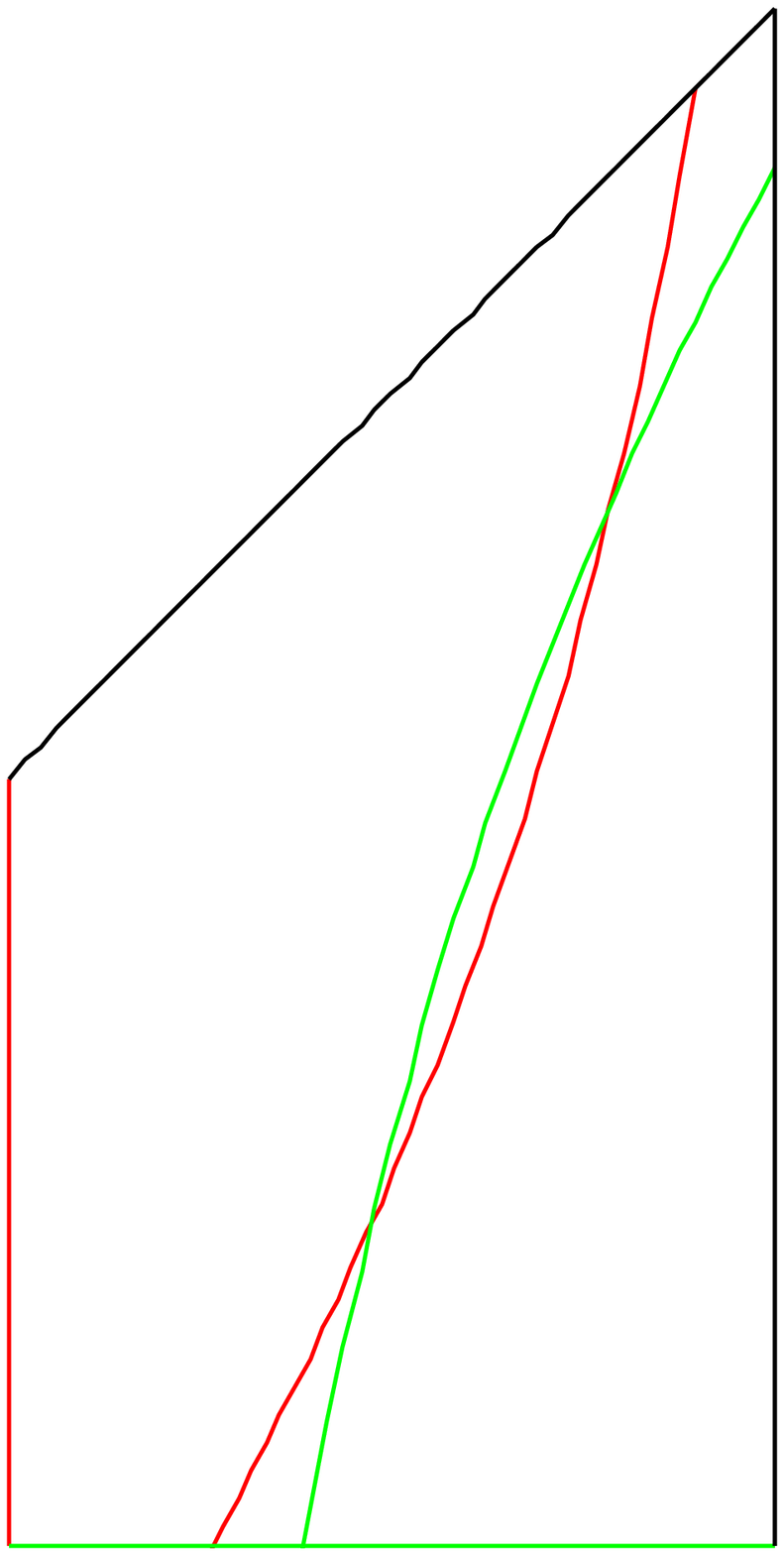}}
\put(9,0.5){\includegraphics[scale=0.2]{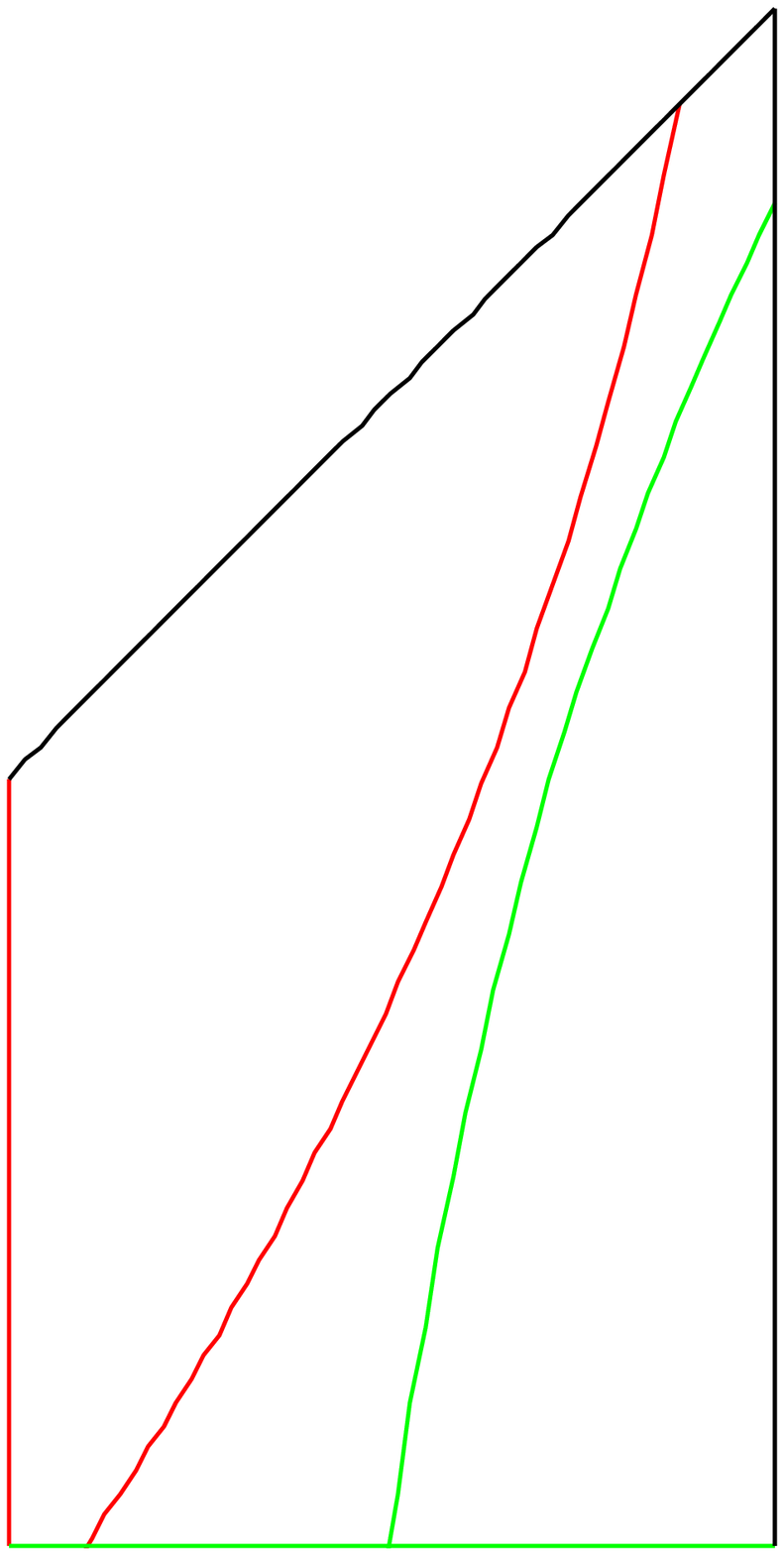}}
\put(12,0.5){\includegraphics[scale=0.2]{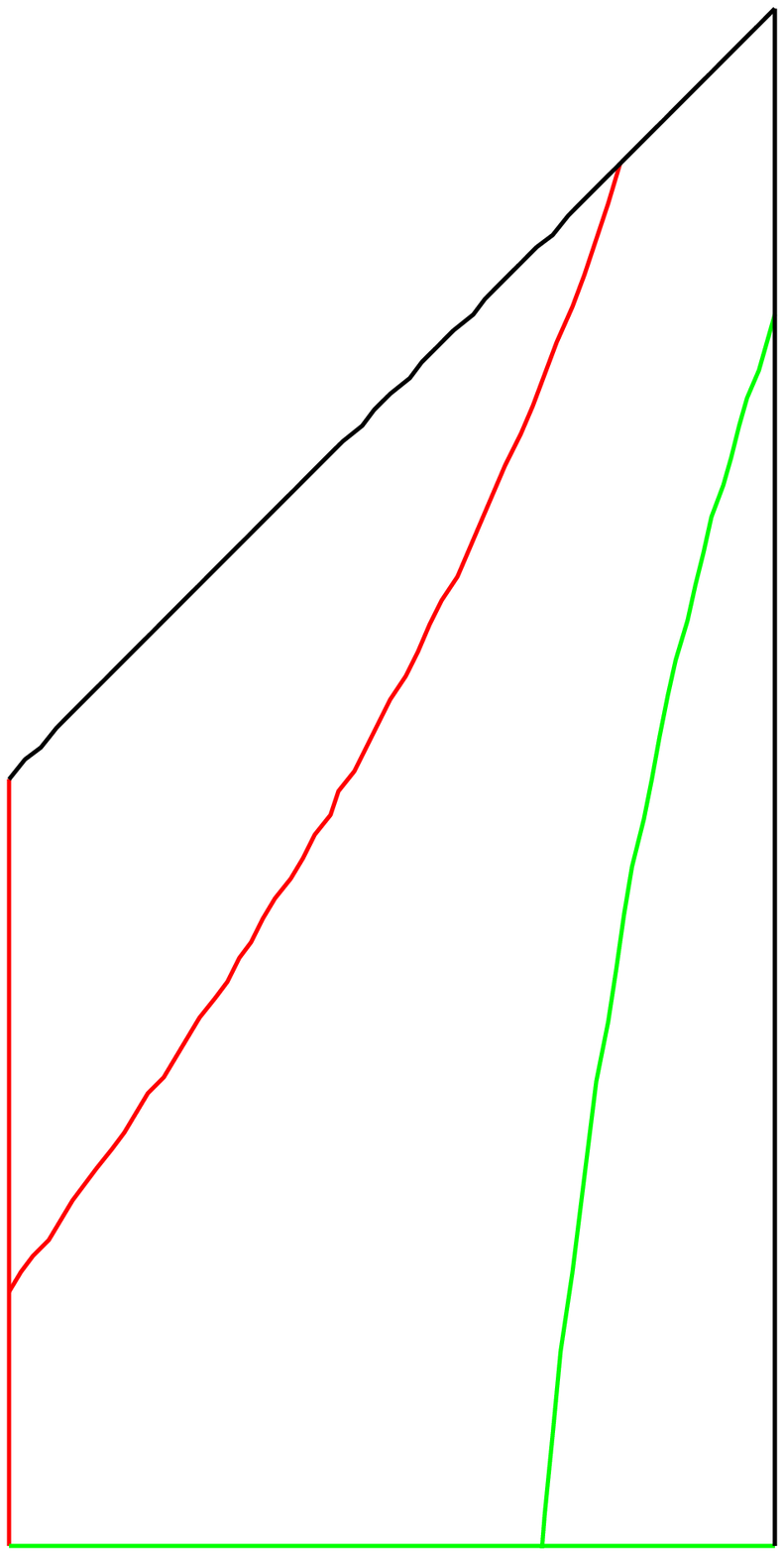}}
\put(2.7,3.92){$\bullet$}
\put(2.3,0.48){$\bullet$}
\put(0.93,1.12){$\bullet$}
\put(0.93,0.48){$\bullet$}
\put(2.3,0.2){{\small$F^1$}}
\put(0.5,1.1){{\small$F^2$}}
\put(2.9,3.9){{\small$F^*$}}
\put(0.9,0.2){{\small$F^0$}}
\put(1.2,-0.2){{\small $D<3/5$}}
\put(5.61,3.62){$\bullet$}
\put(4.83,0.48){$\bullet$}
\put(3.93,0.48){$\bullet$}
\put(4.7,0.2){{\small$F^1$}}
\put(5.9,3.6){{\small$F^*$}}
\put(3.9,0.2){{\small$F^0$}}
\put(3.9,-0.2){{\small $3/5<D<8/9$}}
\put(8.42,3.1){$\bullet$}
\put(7.42,0.48){$\bullet$}
\put(7.82,1.26){$\bullet$}
\put(6.93,0.48){$\bullet$}
\put(7.4,0.2){{\small$F^1$}}
\put(7.4,1.3){{\small$F_1^*$}}
\put(8,3.1){{\small$F_2^*$}}
\put(6.8,0.2){{\small$F^0$}}
\put(7.1,-0.2){{\small $8/9<D<1$}}
\put(10.12,0.48){$\bullet$}
\put(9.93,0.48){$\bullet$}
\put(10.1,0.2){{\small$F^1$}}
\put(9.7,0.2){{\small$F^0$}}
\put(10.1,-0.2){{\small $1<D<6/5$}}
\put(12.93,0.48){$\bullet$}
\put(12.9,0.2){{\small$F^0$}}
\put(13.4,-0.2){{\small $6/5<D$}}
\end{picture}
\end{center}
\caption{Relative positions of the isocline $\dot x_1=0$ (in red) and $\dot x_2=0$ (in green). 
} \label{figisoclines}
\end{figure} 
For instance, consider the following values of the parameters
\begin{equation}\label{parameters}
m_1=8,\quad m_2=4,\quad K_1=L_2=1,\quad L_1=K_2=2,\quad s_1^{in}=s_2^{in}=3
\end{equation}
Then 
$$D_1=6/5,\qquad D_3=8/9,\qquad D_2=3/5.$$
There is another bifurcation value, $D=1$ which correspond to the case when 
the graphs $\Gamma_1$ and $\Gamma_2$ are tangent, see Figure \ref{figbifurcations}.
For this example five cases can occur, see Figure \ref{figisoclines}:
\begin{Prop}
Consider system (\ref{reduit}) where $f_1$ and $f_2$ are given by (\ref{monod}) with parameters (\ref{parameters}). Then
\begin{enumerate}
\item when $D<3/5$, the system has four equilibria, $F^0$ which is an unstable node, $F^1$ and $F^2$, which are saddle points 
and  $F^*$, which is a stable node. This is case (1) of Theorem \ref{th2}, with a unique positive equilibrium.
\item when $3/5<D<8/9$, the system has three equilibria, $F^0$ and $F^1$, which are saddle points 
and  $F^*$, which is a stable node. This is case (2.c) of Theorem \ref{th2}, with a unique positive equilibrium.
\item
when $8/9<D<1$, the system has four equilibria, $F^0$ and $F_1^*$, which are saddle points and $F^1$ 
and $F_2^*$, which are stable nodes. This is case (2.d) of Theorem \ref{th2}, with two positive equilibria.
\item
when $1<D<6/5$, the system has two equilibria, $F^0$, which is a saddle point 
and $F^1$ which is a stable node. 
This is case (2.d) of Theorem \ref{th2}, with no positive equilibrium.
\item
when $D>6/5$, the system has one equilibrium, $F^0$, which is a stable node.
This is case (3) of Theorem \ref{th2}, with no positive equilibrium.
\end{enumerate}
\end{Prop}
\begin{figure}[ht]
\setlength{\unitlength}{1.0cm}
\begin{center}
\begin{picture}(12,4.5)(0,0)
\put(0,0.5){\includegraphics[scale=0.2]{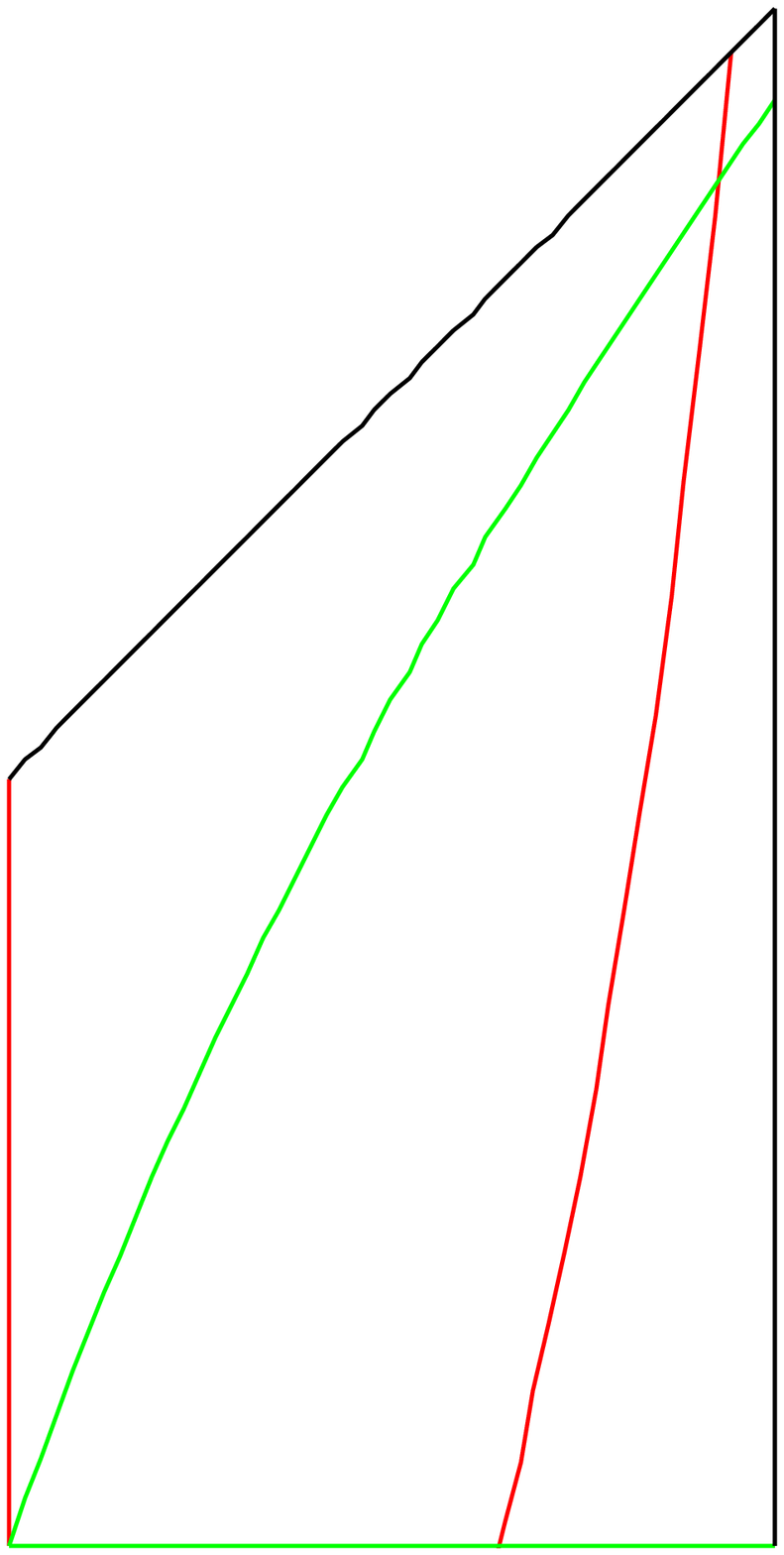}}
\put(3,0.5){\includegraphics[scale=0.2]{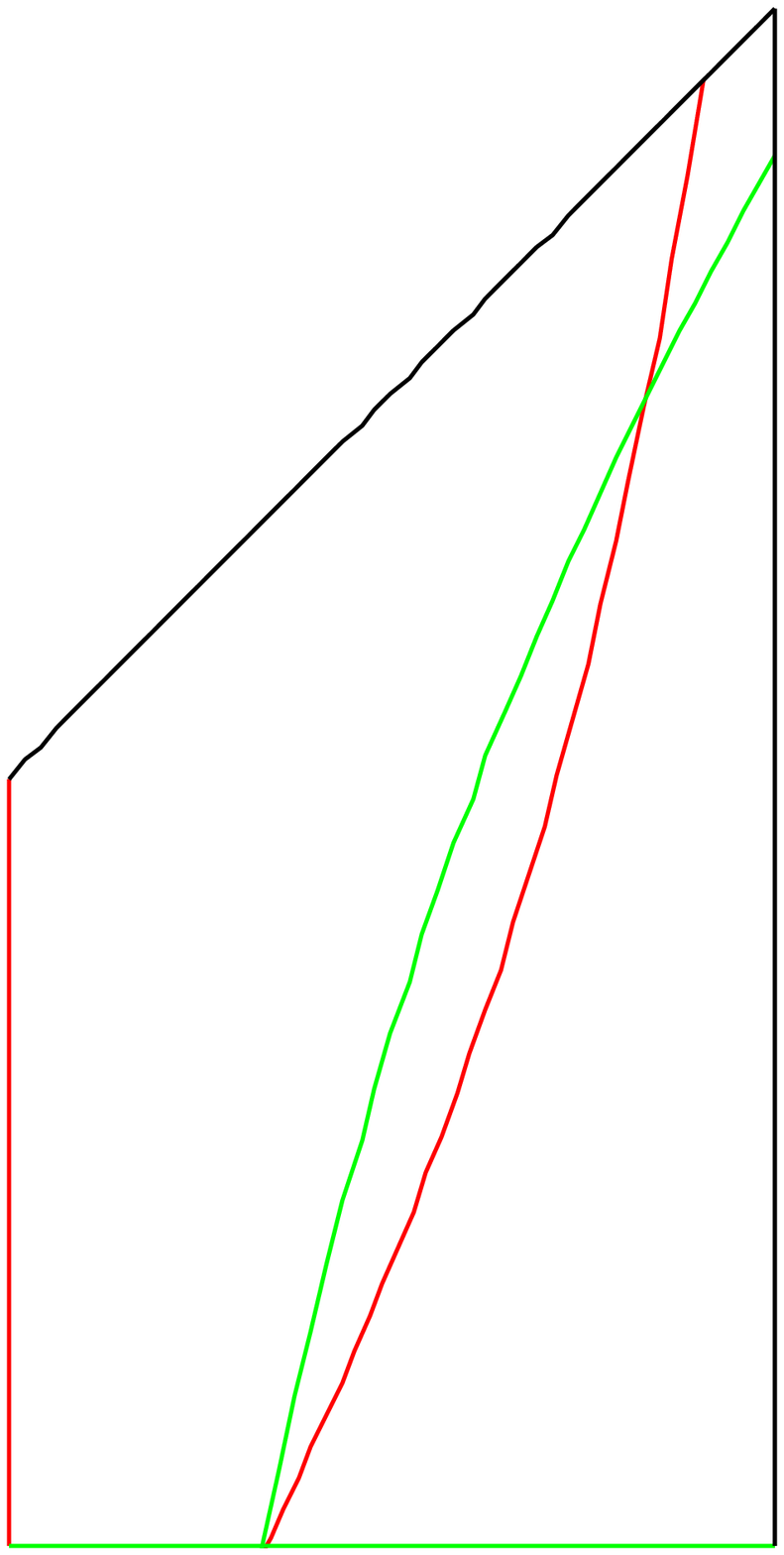}}
\put(6,0.5){\includegraphics[scale=0.2]{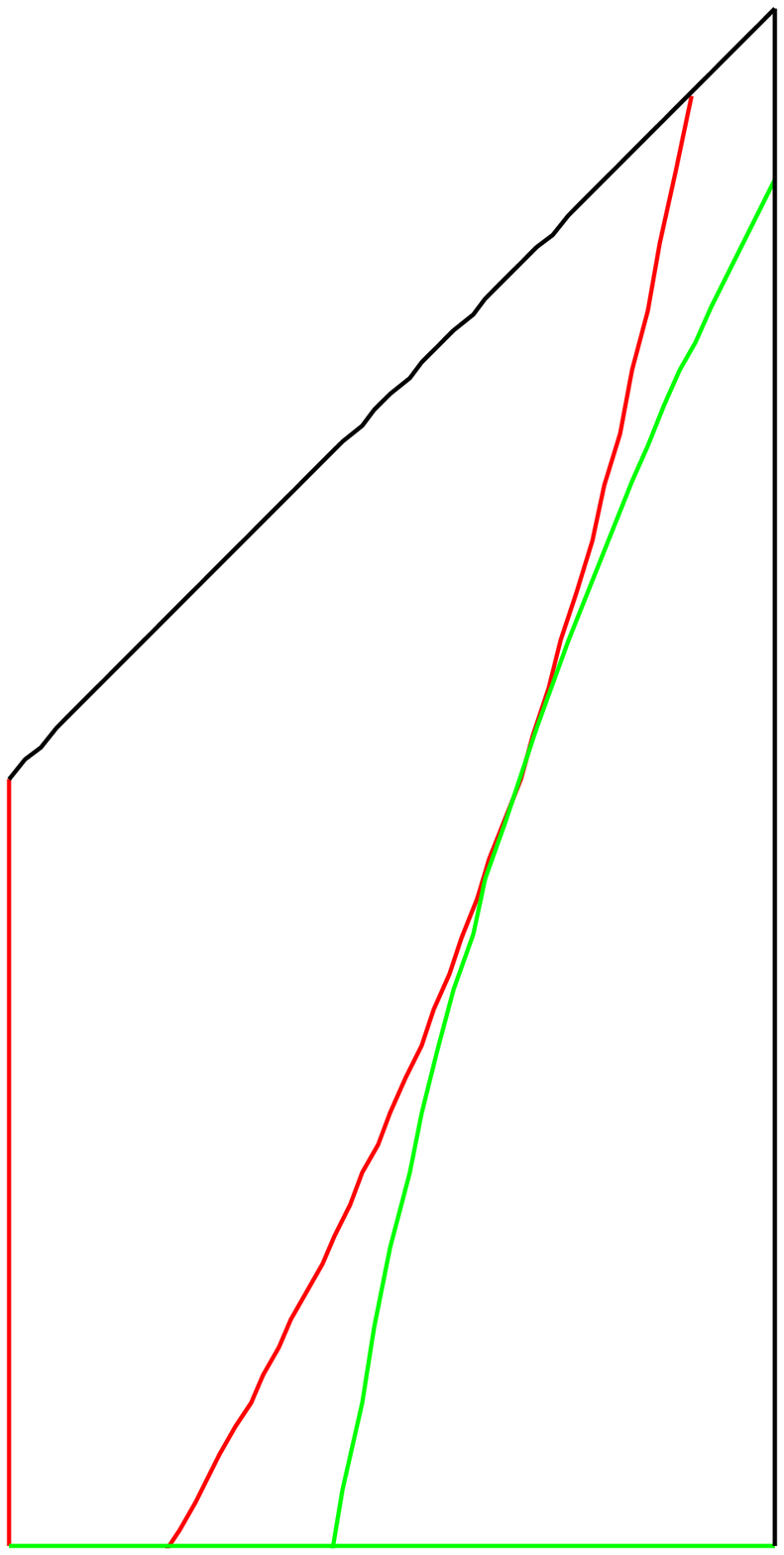}}
\put(9,0.5){\includegraphics[scale=0.2]{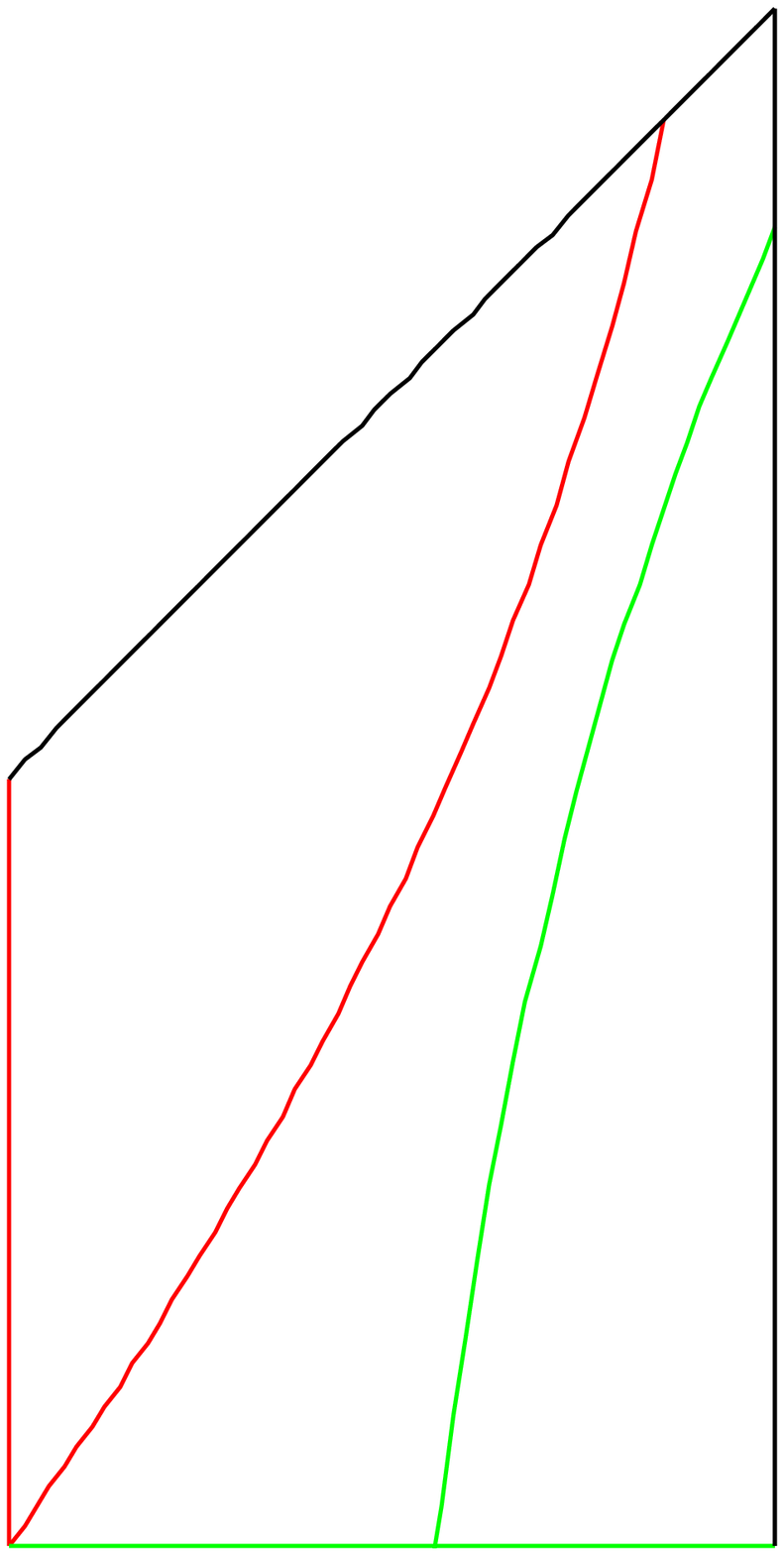}}
\put(2.68,3.83){$\bullet$}
\put(2.15,0.48){$\bullet$}
\put(0.93,0.48){$\bullet$}
\put(2.1,0.2){{\small$F^1$}}
\put(2.9,3.9){{\small$F^*$}}
\put(0.9,0.2){{\small$F^0$}}
\put(1.3,-0.2){{\small $D=3/5$}}
\put(5.5,3.3){$\bullet$}
\put(4.55,0.48){$\bullet$}
\put(3.93,0.48){$\bullet$}
\put(4.5,0.2){{\small$F^1$}}
\put(5.2,3.35){{\small$F^*$}}
\put(3.9,0.2){{\small$F^0$}}
\put(4.3,-0.2){{\small $D=8/9$}}
\put(8.2,2.4){$\bullet$}
\put(7.35,0.48){$\bullet$}
\put(6.93,0.48){$\bullet$}
\put(7.3,0.2){{\small$F^1$}}
\put(7.85,2.4){{\small$F^*$}}
\put(6.8,0.2){{\small$F^0$}}
\put(7.5,-0.2){{\small $D=1$}}
\put(9.93,0.48){$\bullet$}
\put(9.9,0.2){{\small$F^0$}}
\put(10.3,-0.2){{\small $D=6/5$}}
\end{picture}
\end{center}
\caption{The non hyperbolic cases.
When $D=6/5$,  $F^0$ and $F^2$ coalesce.
When $D=8/9$, $F_1^*$ and $F^1$ coalesce (saddle node bifurcation). 
When $D=1$, $F_1^*$ and $F_2^*$ coalesce (saddle node bifurcation).
When $D=6/5$, $F^0$ and $F^2$ coalesce.
 } \label{figbifurcations}
\end{figure} 
In the case when $8/9<D<1$ a bistability phenomenon occurs. According to the initial condition, both species can coexist at equilibrium $F_2^*$, or species $x_2$ goes to extinction at equilibrium $F^1$. This phenomenon  
is illustarted numerically with $D=0.95$ in Figure \ref{bistability}.
\begin{figure}[ht]
\setlength{\unitlength}{1.0cm}
\begin{center}
\begin{picture}(9,4)(0,0)
\put(0,0){\includegraphics[scale=0.2]{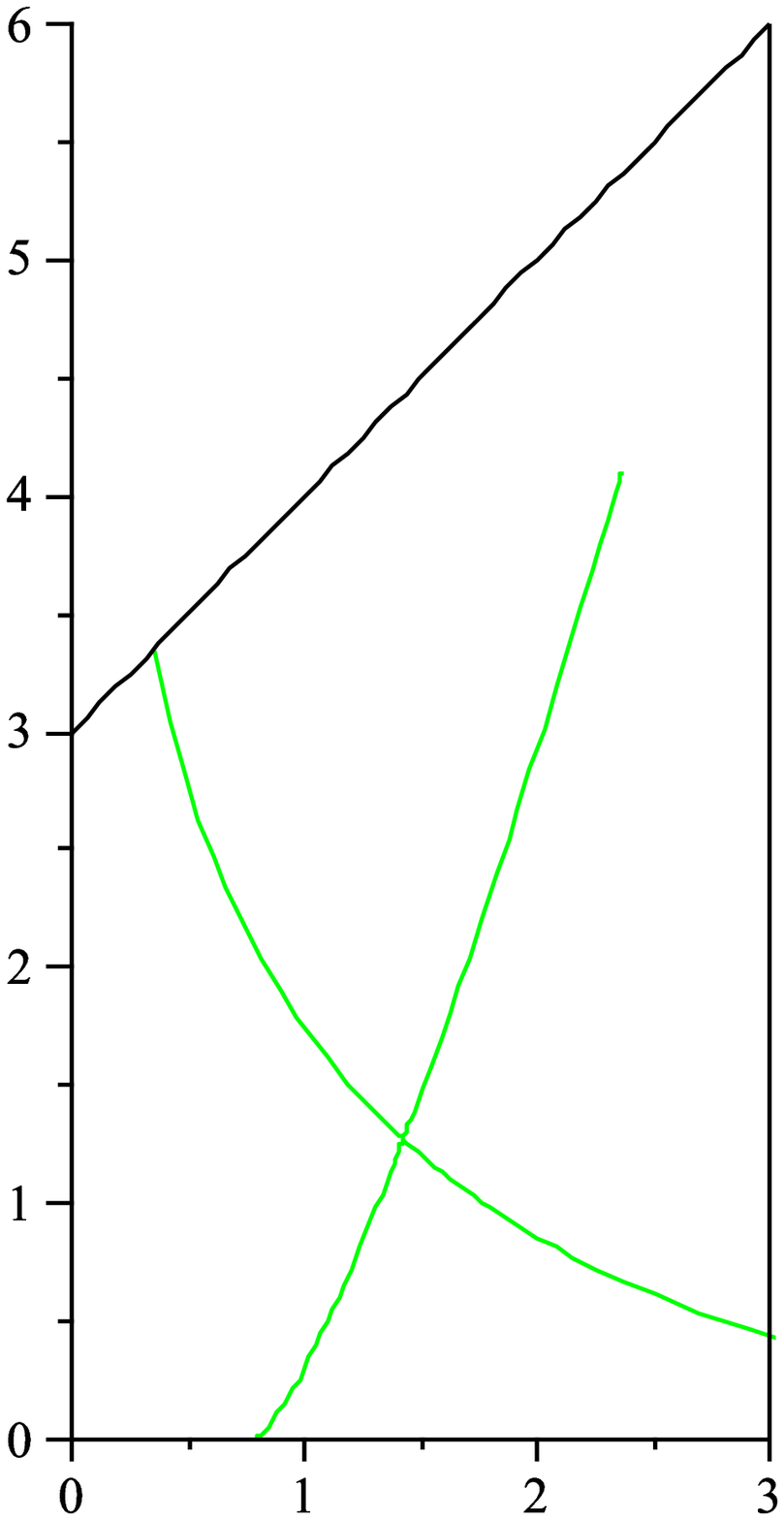}}
\put(3,0){\includegraphics[scale=0.2]{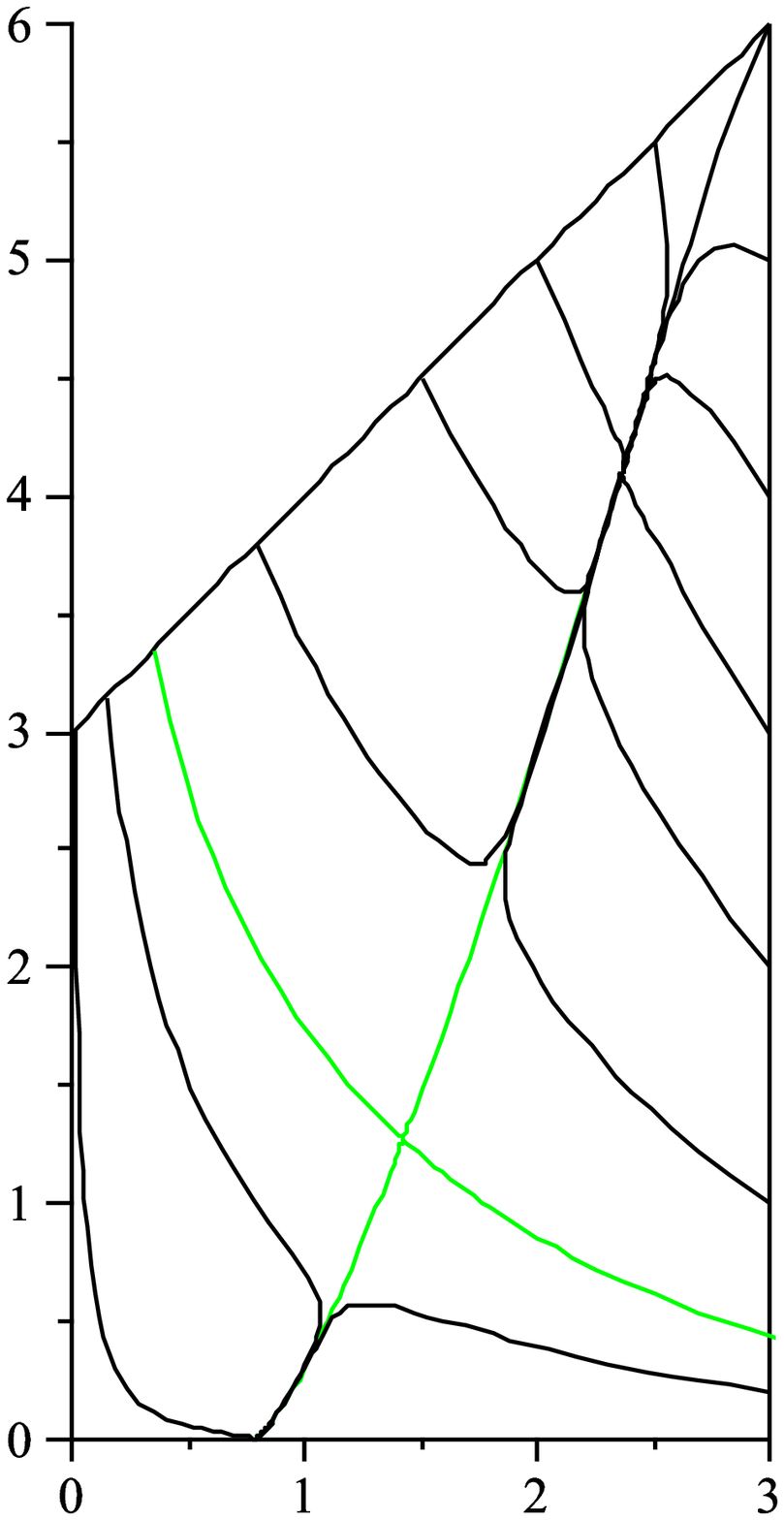}}
\put(6,0.25){\includegraphics[scale=0.185]{iso095.eps}}
\put(1.3,0.3){{\small$F^1$}}
\put(2,1){{\small$F_1^*$}}
\put(2.3,2.8){{\small$F_2^*$}}
\put(7.2,0){{\small$F^1$}}
\put(7.35,1.1){{\small$F_1^*$}}
\put(7.95,2.9){{\small$F_2^*$}}
\put(6.8,0){{\small$F^0$}}
\end{picture}
\end{center}
\caption{Numerical solutions in the bistability case $D=0.95$ and parameters values (\ref{parameters}). 
On the left, the separatrix (in green) of the saddle point $F_1^*$ separate
the domain $\mathcal{S}$ in two region which are the basins of attraction of the boundary equilibrium point $F^1$ and
the positive equilibrium point $F_2^*$. On the center, the phase portrait. On the right, the isoclines.} \label{bistability}
\end{figure} 

For the following values of the parameters
\begin{equation}\label{parameters1}
m_1=8,\quad m_2=7,\quad K_1=K_2=L_2=1,\quad L_1=3/2,\quad s_1^{in}=s_2^{in}=3
\end{equation}
the bifurcational values are  $D_1=4/3$ and $D_2=21/16$. If $D>\max(D_1,D_2)$, for instance for $D=3/2$, 
one obtains a bistability phenomenon corresponding to case (3) of Theorem \ref{th2},
with two positive equilibria. 
 According to the initial condition, both species can coexist at 
equilibrium $F_2^*$, or both species go to extinction at equilibrium $F^0$. 
This phenomenon  
is illustarted numerically in Figure \ref{bistabilitybelle}.
\begin{figure}[ht]
\setlength{\unitlength}{1.0cm}
\begin{center}
\begin{picture}(9,4)(0,0)
\put(0,0){\includegraphics[scale=0.2]{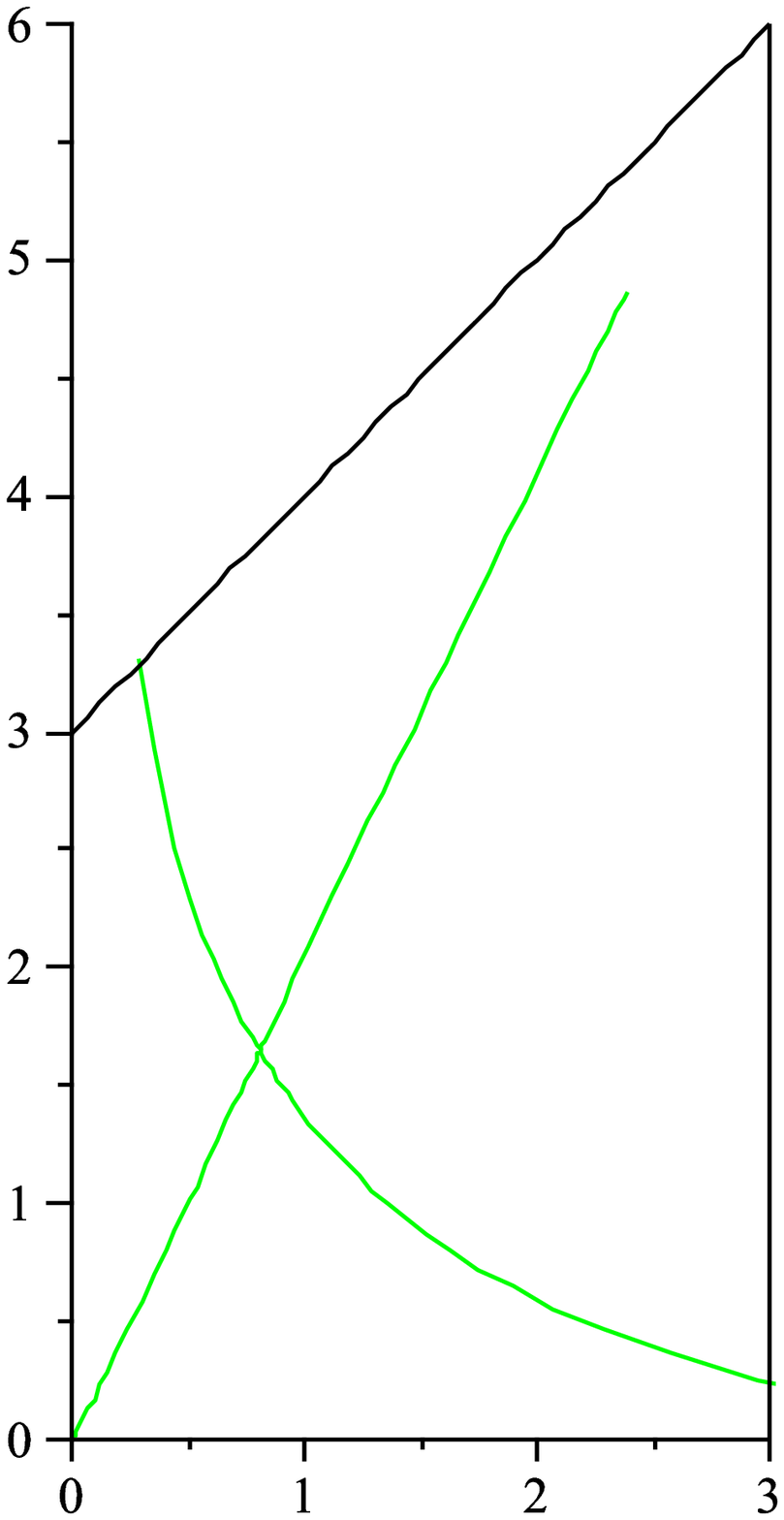}}
\put(3,0){\includegraphics[scale=0.2]{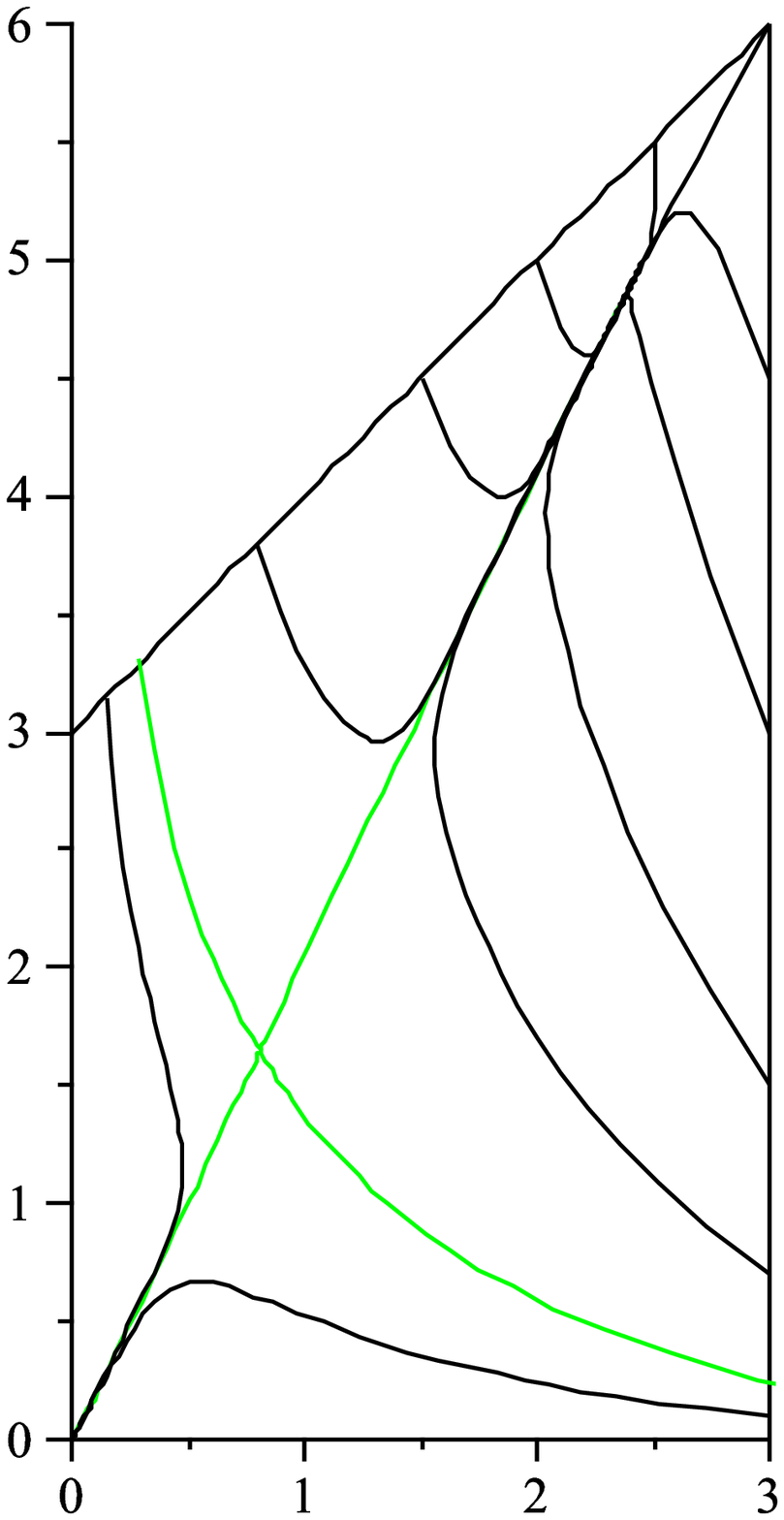}}
\put(6,0.25){\includegraphics[scale=0.185]{isobelle.eps}}
\put(1.2,0.3){{\small$F^0$}}
\put(1.6,1.15){{\small$F_1^*$}}
\put(2.4,2.8){{\small$F_2^*$}}
\put(7.4,1.1){{\small$F_1^*$}}
\put(8.25,2.8){{\small$F_2^*$}}
\put(6.7,0){{\small$F^0$}}
\end{picture}
\end{center}
\caption{Numerical solutions in the bistability case $D=1.5$ and parameters values (\ref{parameters1}). 
On the left, 
the separatrix (in green) of the saddle point $F_1^*$ separate
the domain $\mathcal{S}$ in two region which are the basins of attraction of the boundary equilibrium point $F^0$ and
the positive equilibrium point $F_2^*$. On the center, the phase portrait. On the right, the isoclines.} \label{bistabilitybelle}
\end{figure} 
\section{Global analysis}\label{sec2.3}
Let us establish first that (\ref{reduit}) admits no periodic orbit nor polycycle inside $\mathcal{S}$
\begin{Th}
\label{th1}
\noindent There are no periodic orbits nor polycycles inside $\mathcal{S}$.
\end{Th}
\begin{proof}.
\noindent Consider a trajectory of (\ref{reduit}) belonging to ${\mathcal S}$.
Let us transform the system (\ref{reduit}) through the change of variables
$\xi_1 = \ln(x_1)$, $\xi_2 = \ln(x_2)$.
Then one obtains the following system :
\begin{equation}
\left\{
\begin{array}{lll}
\dot{\xi_1} & =& h_1(\xi_1,\xi_2):=
f_1(s_1^{in}-e^{\xi_1},s_2^{in}+e^{\xi_1}\;-\,e^{\xi_2}) - D,\\
&&\\
\dot{\xi_2} & =& h_1(\xi_1,\xi_2):=
f_2(s_1^{in}-e^{\xi_1},s_2^{in}+e^{\xi_1}\;-\,e^{\xi_2}) - D.
\end{array} \label{orbit}
\right.
\end{equation}
We have
$$
\displaystyle \frac{\partial h_1}{\partial \xi_1}
+ \frac{\partial h_2}{\partial \xi_2}
= \displaystyle - e^{\xi_1 }\frac{\partial f_1}{\partial s_1}+ e^{\xi_1 }\frac{\partial f_1}{\partial s_2}
- e^{\xi_2 }\frac{\partial f_2}{\partial s_2} < 0.
$$
From Dulac criterion \cite{smithbook},
we deduce that the system (\ref{orbit}) has no periodic trajectory.
Hence (\ref{reduit}) has no periodic orbit in ${\mathcal S}$. 
\end{proof}

\begin{Th}
Assume that system (\ref{reduit}) has at most one positive quilibrium $F^*$, then 
for every initial condition in ${\mathcal S}$, the trajectories of system (\ref{reduit}) converge asymptotically to :
\begin{itemize} 
\item $F^*$ if $D<\min(D_1,D_2)$.
\item $F^*$ if $D_1<D_2$ and $D_1<D<D_4$
\item $F^2$ if $D_1<D_2$ and $D_4<D<D_2$.
\item $F^*$ if $D_2<D_1$ and $D_2<D<D_3$.
\item $F^1$ if $D_2<D_1$ and $D_3<D<D_1$.
\item $F^0$ if $\max(D_1,D_4)<D$.
\end{itemize}
\label{th5}
\end{Th}
\begin{proof} 
We restrict the proof to the situation where $D<\min(D_1,D_2)$. The other cases can be done similarly. Let $x_1(0)>0, x_2(0)>0$ and $\omega$ the $\omega$-limit set of $(x_1(0), x_2(0))$. $\omega$ is an invariant compact set and $\omega\subset\bar{\mathcal S}$. Assume that $\omega$ contains a point $M$ on the $x_1x_2$ axis :
\begin{itemize}
\item $M$ can't be $F^0$ because $F^0$ is an unstable node and can't be a part of the $\omega$-limit set of $(x_1(0), x_2(0))$,
\item If $M\in ]\bar x_1,s_1^{in}]\times \{0\}$  (respectively $M\in \{0\}\times ]\bar x_2,s_2^{in}]$). As $\omega$ is invariant then $\gamma(M)\subset\omega$ which is impossible because $\omega$ is bounded and $\gamma(M)=]\bar x_1,+\infty[\times \{0\}$ (respectively $\gamma(M)=\{0\}\times]\bar x_2,+\infty[$),
\item If $M\in]0,\bar x_1[\times \{0\}$ (respectively $M\in\{0\}\ ,\times ]0,\bar x_2[$). $\omega$ contains $\gamma(M)=]0,\bar x_1[\times \{0\}$ (respectively $\gamma(M)=\{0\}\times ]0,\bar x_2[$). As $\omega$ is a compact, then it contains the adherence of $\gamma(M)$, $[0,\bar x_1]\times \{0\}$  (respectively $\{0\}\times [0,\bar x_2]$). In particular, $\omega$ contains $F^0$ which is impossible,
\item If $M=F^1$ (respectively $M=F^2$). $\omega$ is not reduced to $F^1$ (respectively to $F^2$). By Butler-McGehee theorem, $\omega$ contains a point $P$ of $(0,+\infty)\times \{0\}$ other that $F^1$ (respectively of $\{0\}\times (0,+\infty)$ other that $F^2$) which is impossible.
\end{itemize}
Finally, the $\omega$-limit set don't contain any point on the $x_1x_2$ axis. System (\ref{reduit}) has no periodic orbit inside ${\mathcal S}$. Using the Poincaré-Bendixon Theorem \cite{smithbook}, $F^*$ is a globally asymptotically stable equilibrium point for system (\ref{reduit}).
\end{proof}

\begin{Th}
Assume that system (\ref{reduit}) has at most one positive quilibrium $F^*$, then 
for every initial condition in $\mathbb{R}_+^4$, the trajectories of system (\ref{model}) converge asymptotically to:
\begin{itemize}
\item $E^*$ if $D<\min(D_1,D_4)$.
\item $E^*$ if $D_1<D_2$ and $D_1<D<D_4$.
\item $E^2$ if $D_1<D_2$ and $D_4<D<D_2$.
\item $E^*$ if $D_2<D_1$ and $D_2<D<D_3$.
\item $E^1$ if $D_2<D_1$ and $D_3<D<D_1$.
\item $E^0$ if $\max(D_1,D_2)<D$.
\end{itemize}
\label{th3}
\end{Th}
\begin{proof}
Let $(s_1(t),x_1(t),s_2(t),x_2(t))$ be a solution of (\ref{model}).
From (\ref{dz_1/dt}) and (\ref{dz_2/dt}) we deduce that
$$
s_1(t)=s_1^{in}-x_1(t)+K_1e^{-D t}
\quad\mbox{and}\quad
s_2(t)=s_2^{in}+x_1(t)-x_2(t)+K_2e^{-D t},
$$
where $K_1=s_1(0)+x_{1}(0)-s_1^{in}$ and $K_2=s_2(0)+x_2(0)-x_1(0)-s_2^{in}$.
Hence $(x_1(t),x_2(t))$ is a solution of the nonautonomous system of two differential equations :
\begin{eqnarray}
\left\{
\begin{array}{rcl}
\displaystyle  \dot x_1& =&
\displaystyle \left[f_1\left(s_1^{in}-x_1+K_1e^{-D t},s_2^{in}+x_1-x_2+K_2e^{-D t}\right) - D\right] x_1,\\
&&\\
\displaystyle  \dot x_2 &=&
\displaystyle \left[f_2\left(s_1^{in}-x_1+K_1e^{-D t},s_2^{in}+x_1-x_2+K_2e^{-D t}\right) - D\right] x_2.
\end{array}
\label{asymptotiquementautonome}
\right.
\end{eqnarray}
This is an asymptotically autonomous differential system which converge to the autonomous system (\ref{reduit}).
The set $\Omega$ is attractor of all trajectories in $\mathbb{R}_+^4$ and the phase portrait of system reduced to $\Omega$  (\ref{reduit}) contains only locally stable nodes, unstable nodes, saddle points and no trajectory joining two saddle points. Thus we can apply Thiemes's results \cite{thieme} and conclude that the asymptotic behaviour of the solution of the complete system (\ref{asymptotiquementautonome}) is the same that  the asymptotic behaviour  described for the reduced system   (\ref{reduit}) and the main result is then deduced.
\end{proof}

\section{The anaerobic digestion process : An example of a synthrophic
relationship}\label{sec4}
\begin{center}
\begin{figure}[!ht]
\begin{center}
\psset{unit=1.0125cm}
\begin{pspicture}(0,-8)(10,9.5)
  \psline(5,8.5)(5,6.25)
  \put(5,8.5){\vector(0,-1){1.75}}
  \psline(5,6.5)(5,4.25)
  \psline(5,2.5)(5,4)
  \psset{fillstyle=solid}\psline[linewidth=1mm, linecolor=green](5,4)(5,2.5)
  \put(5.2,3.65){\vector(0,-1){0.75}}
  \psset{fillstyle=solid}\psline[linewidth=1mm, linecolor=green](5,2.5)(9,2.5)
  \put(6,2.25){\vector(1,0){2}}
  \psline(5,2.5)(1,2.5)(1,0.5)
  \psset{fillstyle=solid}\psline[linewidth=1mm, linecolor=green](9,2.5)(9,0.5)
  \psline(5,5)(1,5)(1,2.5)
  \psline[linewidth=1mm, linecolor=green](5,5)(9,5)(9,2.5)
  \rput(5,8.5){\psframebox{Macro-molecules}}
  \rput(8.5,7.5){\bf{Hydrolytic-acidogenic}}
  \rput(8.5,7.25){\bf{bacteria}}
  \put(5,7.5){\vector(1,0){1.25}}
  \rput(5,6.5){\psframebox{Monomers}}
  \rput(8.5,5.5){\bf{Acidogenic bacteria}}
  \put(5,5.5){\vector(1,0){1.25}}
  \rput(5,4){\psframebox{{\textcolor{blue}{\bf{Volatile fatty acid}}}}}

\rput(5,2.7){{\textcolor{blue}{\bf{Acetogenic}}}-{\textcolor{blue}{\bf{homoacetogenic}}}}
  \rput(5,2.3){{\textcolor{blue}{\bf{bacteria}}}}
  \put(5,2.5){\vector(-1,0){3.5}}
  \rput[Bl](0,2.5){\psframebox{Acetate}}
  \put(1,4.5){\vector(0,-1){1.65}}
  \put(1,2.35){\vector(0,-1){1.5}}
  \put(9.2,4.75){\vector(0,-1){1.65}}
  \put(9.2,2.2){\vector(0,-1){1.2}}
  \rput[Br](10,2.5){\psframebox{CO$_2$+{\textcolor{blue}{\bf{H$_2$}}}}}
  \rput[Bl](0,0.5){\psframebox{CH$_4$+CO$_2$}}

\rput(6,1.6){{\textcolor{blue}{\bf{Hydrogenotrophic}}}-{\textcolor{blue}{\bf{methanogenic}}}}
  \rput(6,1.25){{\textcolor{blue}{\bf{bacteria}}}}
  \put(8.9,1.35){\vector(-1,0){1}}
  \rput[Br](9.75,0.5){\psframebox{CH$_4$}}

\rput(5,-0.7){{\bf{Considered reactional part}}}
  \psline(2.75,-0.9)(7.25,-0.9)
  \psline(3.25,-0.95)(6.75,-0.95)
  \psline(4.25,-1)(5.75,-1)

  \rput(5,-2){\psframebox{{\textcolor{blue}{\bf{$S_{in}$ : Volatile fatty acid}}}}}
  \rput(8.5,-2){\psframebox{{\textcolor{blue}{\bf{$P_{in}$ : H$_2$}}}}}
  \psline[linewidth=1mm, linecolor=green](5,-2.25)(5,-5)
  \put(5.2,-2.5){\vector(0,-1){2}}
  \rput(5.5,-3.5){{\textcolor{blue}{\bf{$X_1$}}}}
  \psline[linewidth=1mm, linecolor=green](7.7,-2.3)(5.32,-4.8)
  \put(7.7,-2.7){\vector(-1,-1){1}}
  \rput(5,-5){\psframebox{{\textcolor{blue}{\bf{$P$ : H$_2$}}}}}
  \psline[linewidth=0.75mm, linecolor=red](2.4,-2)(2,-2)(2,-6.4)(4.8,-6.4)
  \put(6.25,-3.25){\vector(-1,0){0.5}}
  \rput[Bl](6.35,-4.35){Hydrogen inhibits}
  \rput[Bl](6.35,-4.75){acetogens growth}
  \psline[linewidth=0.75mm, linecolor=red](5.75,-5.1)(6.1,-5.1)(6.1,-3.5)(5.75,-3.5)
  \put(3,-6.25){\vector(2,0){1}}
  \rput[Bl]{90}(2.9,-5.7){V.F.A inhibits}
  \rput[Bl]{90}(3.3,-6.1){methanogens growth}
  \psline[linewidth=1mm, linecolor=green](5,-5.3)(5,-8)
  \put(5.2,-5.5){\vector(0,-1){2}}
  \rput(5.5,-6.5){{\textcolor{blue}{\bf{$X_2$}}}}
  \rput(5,-8){\psframebox{CH$_4$}}

\end{pspicture}
\caption{Anaerobic fermentation process}
\label{fig1}
\end{center}
\end{figure}
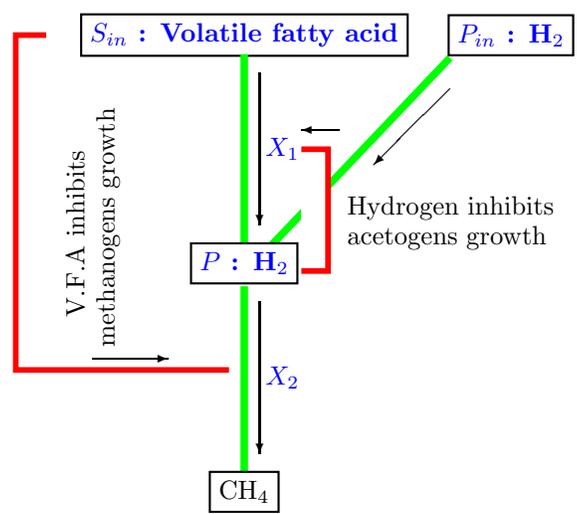
\end{center}
``Methane fermentation'' or ``anaerobic digestion'' is a process that converts organic
matter into a gaseous mixture mainly composed of methane and carbon
dioxide (CH$_4$ and CO$_2$) through the action of a complex bacterial ecosystem (cf.
Fig.\ref{fig1}). It is often used for the
treatment of concentrated wastewaters or to stabilize the excess sludge produced
in wastewater treatment plants into more stable products. There is also considerable
interest in plant-biomass-fed digesters, since the produced methane can be valorized
as a source of energy. It is usually considered that a number of metabolic groups
of bacteria are involved sequencially.

One specific characteristic of the anaerobic process is that within such groups, there
exists
populations exhibiting obligatory mutualistic relationships. Such
a syntrophic relationship is
necessary for the biological reactions to be thermodynamically possible.
In the first steps of the reactions (called ``acidogenesis''), some hydrogen is produced. In El Hajji et al.\cite{elhajji3}, this production of hydrogen at this reaction step was neglected (compare Fig.\ref{fig1} with Fig.{\red{1}} of \cite{elhajji3}). This hypothesis constitue the first novelty with respect to  \cite{elhajji3}.
It is to be noticed that an excess of hydrogen in the medium inhibits the growth of
another bacterial group called ``acetogenic bacteria''. Their association with H$_2$
consuming bacteria
is thus necessary for the second step of the reaction to be fulfilled.
Such a syntrophic relationship has been pointed out in a number of
experimental works (cf. for instance the seminal work by \cite{bryant2}).
Let us consider the subsystem of the anaerobic system where the VFA (for Volatile Fatty
Acids) are transformed into
$H_2$, $CH_4$ and $CO2$. We can
formalize the corresponding biological reactions as a first bacterial consortium $X_1$
(the acetogens) transforming $S_1$ (the VFA) into $S_2$ (the hydrogen) and acetate (cf. Fig.\ref{fig1}).
Then, a second species $X_2$ (the hydrogenotrophic-methanogenic bacteria) grows on $S_2$.
In practice, acetogens are inhibited by an excess of hydrogen and methanogens by an
excess of VFA. Thus, it is further assumed that $X_1$ is inhibitied by $S_2$ and $X_2$ by
$S_1$. The last inhibition relationship constiute the second novelty with respect to \cite{elhajji3}. This situation is precisely the one considered within the model (\ref{model0}).

We have proposed a mathematical model involving a syntrophic relationship of two bacteria. 
It results from this analysis that, under general and natural assumptions of monotonicity on
the functional responses, the stable asymptotic coexistence of the two bacteria is possible. 
\bigskip

\section*{Acknowledgements} The authors acknowledge Inra and Inria for
financial support.

\end{document}